\documentclass[a4paper,11pt,reqno]{amsart}

\usepackage[
textwidth=15cm,
textheight=24cm,
hmarginratio=1:1,
vmarginratio=1:1]{geometry}

\usepackage{graphicx,color,amscd,amsmath,amssymb,verbatim}
\usepackage{amssymb}
\usepackage[english]{babel}
\newtheorem{thm}{Theorem}[section]
\newtheorem{lem}[thm]{Lemma}\newtheorem{prop}[thm]{Proposition}
\theoremstyle{definition}
 \theoremstyle{remark}

\numberwithin{equation}{section}

\newtheoremstyle{remboldstyle}
   {}{}{}{}{\bfseries}{.}{.5em}{{\thmname{#1 }}{\thmnumber{#2}}{\thmnote{ (#3)}}}
\theoremstyle{remboldstyle}

\theoremstyle{remboldstyle}
\newtheorem{examplebold}{Example}[section]

\newcommand{\D}{\mathbb D}
\newcommand{\C}{\mathbb C}

\theoremstyle{remark}

\numberwithin{equation}{section}
\numberwithin{theorem}{section}

\begin{document}

\title[]{On the exact value of the norm of the Hilbert matrix operator on weighted Bergman spaces}

\author[Lindstr{\"o}m] {Mikael Lindstr\"om}\address{Mikael Lindstr{\"o}m. Department of Mathematics, \AA bo Akademi University. FI-20500 \AA bo, Finland. \emph{e}.mail: mikael.lindstrom@abo.fi}
\author[Miihkinen] {Santeri Miihkinen} \address{Santeri Miihkinen. Department of Mathematics, \AA bo Akademi University. FI-20500 \AA bo, Finland. \emph{e}.mail: santeri.miihkinen@abo.fi}
\author[Wikman] {Niklas Wikman}\address{Niklas Wikman. Department of Mathematics, \AA bo Akademi University. FI-20500 \AA bo, Finland. \emph{e}.mail: niklas.wikman@abo.fi}

\begin{abstract} 
In this article, the open problem of finding the exact value of the norm of the Hilbert matrix operator on weighted Bergman spaces $A^p_\alpha$ is adressed. The norm was conjectured to be $\frac{\pi}{\sin \frac{(2+\alpha)\pi}{p}}$ by Karapetrovi\'{c}. We obtain a complete solution to the conjecture for $\alpha \ge 0$ and $2+\alpha+\sqrt{\alpha^2+\frac{7}{2}\alpha+3} \le p < 2(2+\alpha)$ and a partial solution for $2+2\alpha < p < 2+\alpha+\sqrt{\alpha^2+\frac{7}{2}\alpha+3}.$ Moreover, we also show that the conjecture is valid for small values of $\alpha$ when $2+2\alpha < p \le 3+2\alpha.$ Finally, the case $\alpha = 1$ is considered.
\end{abstract}



\maketitle
\section{Introduction}

The Hilbert matrix operator $\mathcal{H}$ is a linear integral operator that can be defined on several spaces of analytic functions on the open unit disk $\mathbb{D} = \{ z \in \mathbb{C} : |z| < 1\}$. 
Historically, Magnus \cite{3} was the first person to consider $\mathcal{H}$ as an operator on the space $\ell^2$ of square-summable complex sequences. %and investigated its spectral properties. 
%Historically, the Hilbert matrix $\mathcal{H}$ was considered  first time as an operator on the space $\ell^2$ of square-summable complex sequences by Magnus in \cite{3}, where sider. 
The first results in the direction of estimating the norm of $\mathcal{H}$ on analytic function spaces were obtained by Diamantopoulos and Siskakis in \cite{1, 10}. In \cite{1}, they considered $\mathcal{H}$ acting on the Hardy spaces and established the boundedness of $\mathcal{H}$ for $1 < p < \infty$ in combination with upper estimates for the norm.
Aleman, Montes-Rodr\'{i}guez, and Sarafoleanu  \cite{1} provided a description of the point spectrum of   $\mathcal{H}$  acting on the  Hardy spaces.
%and gave an upper bound for its norm when $2 \le p < \infty$,
%\begin{equation*}
%	\lVert \mathcal{H} \rVert_{H^p \to H^p} \leq \frac{\pi}{\sin \frac{\pi}{p}}
%\end{equation*}
% and a less precise estimate for $1 < p < 2$. 
 The starting point for the study of the Bergman space case was the article \cite{10}, in which  Diamantopoulos established an upper estimate for the norm of $\mathcal{H}$ for $4 \le p < \infty$ and a less precise estimate on the scale $2 < p < 4$. In \cite{7}, Dostani\'{c}, Jevti\'{c} and Vukoti\'{c} pursued the investigation of $\mathcal{H}$ on the Bergman space $A^p$ and Hardy space $H^p$ and obtained the exact value of the norm for $4 \leq p < \infty$, namely
\begin{equation*}
	\lVert \mathcal{H} \rVert_{A^p \to A^p} = \frac{\pi}{\sin \frac{2\pi}{p}}
\end{equation*}
and in the Hardy space case the precise value of the norm was proven to be $$\lVert \mathcal{H} \rVert_{H^p \to H^p}=\frac{\pi}{\sin \frac{\pi}{p}}$$ for $1 < p < \infty.$ 
They also conjectured that the value of the norm $\lVert \mathcal{H} \rVert_{A^p \to A^p} $ is the same %also 
in the $2 < p < 4$ case, see also \cite{12}. Bo\v{z}in and Karapetrovi\'{c} \cite{4} confirmed the conjecture in the positive by reducing the problem to certain novel estimates of the Beta function. In \cite{8} the authors simplified the proofs of the key lemmas in \cite{4} significantly by discarding the use of a classical theorem of Sturm. 

In this article, which is a continuation of \cite{8}, the results concerning the unweighted Bergman space case are generalized to the weighted Bergman spaces $A_{\alpha}^p$, where $\alpha \geq 0$. Karapetrovi\'{c} already considered $\mathcal{H}$ on $A_{\alpha}^p$ in \cite{6}, where he derives the exact value of the norm of $\mathcal{H}$ when $4\leq 2(2+\alpha) \leq p < \infty$, that is
\begin{equation*}
	\lVert \mathcal{H} \rVert_{A_{\alpha}^p \to A_{\alpha}^p} = \frac{\pi}{\sin \frac{(2+\alpha)\pi}{p}}
\end{equation*}
and obtains a better than known upper bound for the norm %of $\mathcal{H}$ 
when $2 \leq 2 +\alpha < p < 2(2+\alpha)$. In \cite{6} Karapetrovi\'{c} conjectures that the norm of $\mathcal{H}$ is the same as above also in the case $2 < 2+\alpha \leq p < 2(2+\alpha)$.
In this article the conjecture is confirmed in the positive for $2+\alpha +\sqrt{\alpha^2+\frac{7}{2}\alpha +3} \le p < 2(2+\alpha)$. The main result of the article is as follows.

\begin{thm}\label{Main Theorem} Let $\alpha \geq 0$. Suppose that either of the following conditions holds:
	\begin{enumerate}
		\item[(a)] $2 +\alpha + \sqrt{\alpha^2+\frac{7}{2}\alpha + 3} \le p <2(2 +\alpha)$;
		\item[(b)] $2 + 2\alpha < p < 2 +\alpha + \sqrt{\alpha^2+\frac{7}{2}\alpha + 3}$ and 
		\begin{equation*}
			\int_0^1 I_t\left(\frac{2+\alpha}{p}, 1 - \frac{2+\alpha}{p}\right)t^{2p-4\alpha-5}(1-t^4)^\alpha dt - \frac{1}{4(\alpha +1)} \leq 0,
		\end{equation*}
		where $I_t$ is the  regularized incomplete Beta function.
	\end{enumerate}
	Then $\lVert \mathcal{H} \rVert_{A_{\alpha}^p \to A_{\alpha}^p} = \frac{\pi}{\sin{\frac{(2+\alpha)\pi}{p}}}$. 
\end{thm}

%\vspace{12pt}

The proof of this result is based on two lemmas. In both lemmas, two functions, central to our study, appear. These auxiliary functions are defined via infinite convergent sums and arise from the evaluation of an involved integral via the use of series expansions. In the first lemma, an estimate for the Beta function is established and it confirms the conjecture for the large values of $p$, namely for $2+ \alpha +\sqrt{\alpha^2+\frac{7}{2}\alpha +3} \le p < 2(2+\alpha)$. In the second lemma, which is a generalization of Lemma 2.6 in \cite{4} for the weighted Bergman spaces, the values $2+2\alpha < p <2+\alpha+\sqrt{\alpha^2+\frac{7}{2}\alpha +3}$ are considered. %In both lemmas, two functions, central to our study, appear. These auxiliary functions are defined via infinite convergent sums and arise from the evaluation of an involved integral via the use of series expansions. 
Furthermore, a sufficient condition for the conjecture to hold is also introduced in the second lemma (see (b)-part of Theorem \ref{Main Theorem}). %In the case $\alpha = 0$, the condition is equivalent with the estimate for the Beta function presented in Lemma 2.5 in \cite{4}. %In both lemmas, two functions, central to our study, appear. These auxiliary functions are defined via infinite convergent sums and arise from the evaluation of an involved integral via the use of series expansions. Furthermore, the authors find a sufficient condition needed for the conjecture to hold also in the $2+2\alpha < p <2+\alpha+\sqrt{\alpha^2+\frac{7}{2}\alpha +3}$ case. 
It turns out that this condition does not hold for every $\alpha >  0$ when $2+2\alpha < p < 2+ \alpha +\sqrt{\alpha^2+\frac{7}{2}\alpha +3}$. Moreover, it is shown that the conjecture is valid for $0 < \alpha \leq 1/19$, when $2+2\alpha < p < \frac{5}{2}+2\alpha$ and for $0 < \alpha \leq \frac{1}{15000}$ when $\frac{5}{2}+2\alpha \le p < 3 + 2\alpha$. Lastly, the case $\alpha = 1$ is studied as an example and it is observed that the condition does not hold for small values of $p$, but it does hold for large values of $p$.

The article is organized as follows. In section 2, we recall some preliminaries, including the integral representation of the Hilbert matrix operator in terms of certain weighted composition operators and classical identities concerning the Beta and Gamma functions. Section 3 contains auxiliary results such as estimates for the Beta function due to Bhayo, S\'andor and Iv\'ady and definitions of the aforementioned auxiliary functions. The two key lemmas are also presented in this section. In Section 4, the proof of the main result is provided and it is followed by Lemma 4.1 and Proposition 4.2, in which the focus is on the small values of $\alpha.$ Finally, we conclude with the case $\alpha = 1$ in Example 4.1.%, where it is shown that the condition \eqref{Condition} holds for large values of $p$ and does not hold for small values of $p.$ 

\section{Preliminaries}
Let $H(\mathbb{D})$ be the algebra of all analytic functions on the unit disk $\mathbb{D}$. For an analytic function $f$ it holds that $f(z) = \sum_{k = 0}^\infty a_kz^k$, where $a_k \in \C$. The Hilbert matrix operator $\mathcal{H}$ can be expressed as an operator on spaces of analytic functions by its action on the Taylor coefficients $a_k$ in the following way
\begin{equation*}
	\mathcal{H}(f)(z) = \sum_{n = 0}^\infty \left(\sum_{k = 0}^\infty \frac{a_k}{n+k+1}\right)z^n.
\end{equation*}
The operator $\mathcal{H}$ can also be written as an integral average of certain weighted composition operators as follows
\begin{equation*}
	\mathcal{H}(f)(z) =\int_0^1 T_t(f)(z)dt,
\end{equation*}
where $T_t(f)(z) = w_t(z)f(\phi_t(z))$, $w_t(z) = \frac{1}{(t-1)z+1}$ and $\phi_t(z) = \frac{t}{(t-1)z+1}$ for $z \in \D$ and $0 < t < 1$.
The standard weighted Bergman spaces are defined as
\begin{equation*}
	A_{\alpha}^p(\mathbb{D}) := \left\{ f \in H(\mathbb{D}) : \lVert f \rVert_{A_{\alpha}^p} = \left(\int_{\mathbb{D}} |f(w)|^p dA_{\alpha}(w)\right)^{1/p} < \infty \right\},
\end{equation*}
where $dA_{\alpha}(w) = (\alpha+1)(1-|w|^2)^{\alpha}dA(w)$ and $dA(w)$ is the normalized Lebesgue area measure on $\mathbb{D}$. Throughout this article we will assume that $p> 2 +\alpha$ and $ \alpha \geq 0$. We have the following upper estimate: 
\begin{equation}
\label{Normstart1}
	\lVert \mathcal{H}(f)\rVert_{A_{\alpha}^p} \leq \int_0^1 \lVert T_t(f)\rVert_{A_{\alpha}^p} dt.
\end{equation}
Moreover,
\begin{equation}\label{Normstart2}
	\begin{split}
		\lVert T_t(f) \rVert_{A_{\alpha}^p} &= \left(\int_{\mathbb{D}} |T_t(f)(w)|^p dA_{\alpha}(w)\right)^{1/p}\\
								&\leq\frac{t^{\frac{2+\alpha}{p}-1}}{(1-t)^{\frac{2+\alpha}{p}}}\left((\alpha+1)\int_{D_t}|w|^{p-2\alpha-4}|f(w)|^p(1-|w|^2)^{\alpha} dA(w)\right)^{1/p},
	\end{split}
\end{equation}
where $D_t = D(\frac{1}{2-t},\frac{1-t}{2-t})$ is the disk with radius $\frac{1-t}{2-t}$ and center $\frac{1}{2-t}$, see \cite{6}. We will also need the Beta function, which is defined as the integral
\begin{equation*}
	B(s,t) = \int_0^1 t^{s-1}(1-t)^{t-1}dt,
\end{equation*}
where $s$ and $t$ are complex numbers satisfying $\Re(s) > 0$ and $\Re(t) > 0$. It can be checked that $B(s,t) = \frac{\Gamma(s)\Gamma(t)}{\Gamma(s+t)}$, where $\Gamma$ is the Gamma function:
\begin{equation*}
	\Gamma(z) = \int_0^\infty x^{z-1}e^{-x}dx, \quad \Re(z) > 0.
\end{equation*}
We will use the well-known reflection formula 
\begin{equation*}
	\Gamma(z)\Gamma(1-z) = \frac{\pi}{\sin \pi z}, \hspace{0.25cm} z \in \mathbb{C} \setminus \mathbb{Z}.
\end{equation*}
The incomplete Beta function, denoted by $B_t$, is defined as
\begin{equation*}
	B_t(x,y) = \int_0^t s^{x-1}(1-s)^{y-1}ds.
\end{equation*}
By dividing $B_t(x,y)$ with $B(x,y)$ we obtain the regularized incomplete Beta function
\begin{equation*}
	I_t(x,y) = \frac{B_t(x,y)}{B(x,y)}.
\end{equation*}
The binomial series $(1+z)^\alpha = \sum_{k = 0}^\infty \binom{\alpha}{k} z^k$, where $\alpha, z \in \C$ and $\binom{\alpha}{k}$ are the generalized binomial coefficients defined as
\begin{equation*}
	\binom{\alpha}{k} = \frac{\alpha(\alpha-1)\cdots(\alpha -k -1)}{k!}, \quad \binom{\alpha}{0}=1,
\end{equation*}
converges absolutely for all complex values of $\alpha$ when $|z| < 1$. In the context of the article, the parameter $\alpha$ is always a real number. We refer the interested reader to \cite{2} for these and other identities regarding the Beta function and Gamma function.

\section{Auxiliary results}
In this section several lemmas needed for the main result are presented. The following result can be found in \cite{5, 11} and hence its proof is omitted.
\begin{lem}\label{11}\label{lower bound1} Let $x \ge 1$, $0 < y \le 1$. Then
	\begin{enumerate}
		\item[(a)] $B(x,y) \leq \frac{1}{x}+\frac{1}{y} -1;$
		\item[(b)] $B(x,y) \geq \frac{1}{xy}\frac{x+y}{1+xy}$.
	\end{enumerate}
	The inequalities reverse when $x,y \in (0,1]$.
\end{lem}
Let $\alpha \geq 0$ and $2+2\alpha < p < 2(2+\alpha)$. Define the functions
\begin{equation*}
	\begin{split}
		&H_ {\alpha, p}(s) = \sum_{k=0}^\infty \binom{\alpha}{k} (-1)^k \frac{1}{p-2\alpha -2 +2k}
	 - \frac{1}{2(\alpha +1)}(1-s^4)^{\alpha +1}, \hspace{0.25cm} 0 \le s \le 1; \\
		&K_{\alpha, p}(s,t) =\sum_{k=0}^\infty \binom{\alpha}{k} (-1)^k \frac{1}{p-2\alpha -2 +2k}
		\max\{s^2,t^2\}^{p - 2\alpha-2+2k}, \hspace{0.25cm} 0 \le  s \le 1, 0< t < 1.
	\end{split}
\end{equation*} 
Since for every $k \ge 1$ it holds that
\begin{equation*}
	\left|\frac{(-1)^k}{p-2\alpha-2+2k}\right| < 1
\end{equation*}
and
\begin{equation*}
	\left| \frac{(-1)^k}{p-2\alpha -2 +2k}\max\{s^2,t^2\}^{p - 2\alpha-2+2k}\right| < 1,
\end{equation*}
the functions are well defined. These functions will appear in our two key lemmas needed for the main proof. The following expressions will turn out to be important later.
\begin{lem}\label{Condition1} Let $\alpha \geq 0$ and $2+ 2\alpha < p < 2(2 + \alpha)$ and define $\psi_{\alpha,p}(t) = t^{\frac{2+\alpha}{p}-1}(1-t)^{-\frac{2+\alpha}{p}}$. Then the following statements are equivalent:
	\begin{enumerate}
		\item[(a)]
			\begin{equation*}
				B\left(\frac{2+\alpha}{p}, 1 - \frac{2+\alpha}{p}\right)H_{\alpha, p}(0) -\int_0^1 \psi_{\alpha, p}(t)K_{\alpha, p}(0,t)dt \leq 0;
			\end{equation*}
		\item[(b)]
			\begin{equation}\label{Condition(otherform)}
				\begin{split}
					&B\left(\frac{2+\alpha}{p}, 1 - \frac{2+\alpha}{p}\right)\left[\frac{1}{2}B\left(\frac{p-2\alpha-2}{2}, \alpha +1\right) - \frac{1}{2(\alpha+1)}\right]\\
					&-\frac{1}{2}\int_0^1 \psi_{\alpha, p}(t)B_{t^4}\left(\frac{p-2\alpha-2}{2}, \alpha +1\right)dt \leq 0; 
				\end{split}	
			\end{equation}
			%where $\psi_{\alpha,p}(t) = t^{\frac{2+\alpha}{p}-1}(1-t)^{-\frac{2+\alpha}{p}}$;
		\item[(c)] 
			\begin{equation}\label{Condition}
				\int_0^1 I_t\left(\frac{2+\alpha}{p}, 1 - \frac{2+\alpha}{p}\right)t^{2p-4\alpha-5}(1-t^4)^\alpha dt - \frac{1}{4(\alpha +1)} \leq 0.
			\end{equation}
	\end{enumerate}
\end{lem}
\begin{proof}[Proof.]
	Assume that (a) holds. Then
	\begin{equation}\label{LL}
		\begin{split}
			&B\left(\frac{2+\alpha}{p}, 1 - \frac{2+\alpha}{p}\right)\left[\sum_{k = 0}^\infty \binom{\alpha}{k}(-1)^k \frac{1}{p-2\alpha-2+2k} - \frac{1}{2(\alpha+1)}\right]\\
			&-\int_0^1  \psi_{\alpha, p}(t)\left[\sum_{k = 0}^\infty \binom{\alpha}{k}(-1)^k\frac{t^{2(p-2\alpha-2+2k)}}{p -2\alpha-2+2k}\right]dt \leq 0.
		\end{split}
	\end{equation}
	Now we evaluate the sums $\sum_{k = 0}^\infty \binom{\alpha}{k}(-1)^k \frac{1}{p-2\alpha-2+2k}$ and $\sum_{k = 0}^\infty \binom{\alpha}{k}(-1)^k\frac{t^{2(p-2\alpha-2+2k)}}{p -2\alpha-2+2k}$. For the first sum it holds that
	\begin{equation}
	\label{firstsum}
		\sum_{k = 0}^\infty \binom{\alpha}{k}(-1)^k t^{p-2\alpha-3+2k} = t^{p-2\alpha-3}(1-t^2)^\alpha.
	\end{equation}
	Integrating both sides gives us
	\begin{equation*}
		\begin{split}
		%	&\sum_{k = 0}^\infty \binom{\alpha}{k}(-1)^k\left[\frac{t^{p-2\alpha-2+2k}}{p-2\alpha-2+2k}\right]_0^1 = \int_0^1 t^{p-2\alpha-3}(1-t^2)^\alpha dt\\
			&\sum_{k = 0}^\infty \binom{\alpha}{k}(-1)^k\frac{1}{p- 2\alpha-2+2k} =  \int_0^1 t^{p-2\alpha-3}(1-t^2)^\alpha dt,
		\end{split}
	\end{equation*}
	where the order of integration and summation is justified by Lebesgue's dominated convergence theorem.
	For the right-hand side we have
	\begin{equation*}
		\begin{split}
			\int_0^1 t^{p-2\alpha-3}(1-t^2)^\alpha dt &= \int_0^1 s^{\frac{p-2\alpha -3}{2}}(1-s)^\alpha \frac{ds}{2\sqrt{s}}\\
										&=\frac{1}{2}\int_0^1 s^{\frac{p-2\alpha-4}{2}}(1-s)^\alpha ds\\
										&= \frac{1}{2} B\left(\frac{p-2\alpha -2}{2}, \alpha +1\right).
		\end{split}
	\end{equation*}
	So we obtain
	\begin{equation*}
		\sum_{k = 0}^\infty \binom{\alpha}{k}(-1)^k\frac{1}{p-2\alpha-2+2k} = \frac{1}{2}B\left(\frac{p-2\alpha-2}{2}, \alpha +1\right).
	\end{equation*}
	For the second sum we integrate both sides of \eqref{firstsum} from $0$ to $t^2$ and the dominated convergence theorem implies that
%	\begin{equation*}
%		 \sum_{k = 0}^\infty \binom{\alpha}{k}(-1)^k t^{p-2\alpha -3 +2k} = t^{p-2\alpha-3}(1-t^2)^\alpha.
%	\end{equation*}
	%Integrating both sides we get, by the dominated convergence theorem.
	\begin{equation*}
		\begin{split}
		 	&\sum_{k = 0}^\infty \binom{\alpha}{k}(-1)^k \left[\frac{t^{p-2\alpha -2 +2k}}{p-2\alpha-2+2k}\right]_0^{t^2} = \int_0^{t^2} t^{p-2\alpha -3}(1-t^2)^\alpha dt,\\
			&\sum_{k = 0}^\infty \binom{\alpha}{k}(-1)^k\frac{t^{2(p-2\alpha -2 +2k)}}{p -2\alpha -2 +2k} = \int_0^{t^4}s^{\frac{p-2\alpha-3}{2}}(1-s)^\alpha \frac{ds}{2\sqrt{s}} = \frac{1}{2}B_{t^4}\left(\frac{p-2\alpha-2}{2}, \alpha +1\right).
		\end{split}
	\end{equation*}
	Estimate $(\ref{LL})$ now takes the form
	\begin{equation*}
		\begin{split}
			&B\left(\frac{2+\alpha}{p}, 1 - \frac{2+\alpha}{p}\right)\left[\frac{1}{2}B\left(\frac{p-2\alpha-2}{2}, \alpha +1\right) - \frac{1}{2(\alpha+1)}\right]\\
			&-\frac{1}{2}\int_0^1 \psi_{\alpha, p}(t)B_{t^4}\left(\frac{p-2\alpha-2}{2}, \alpha +1\right)dt \leq 0,
		\end{split}
	\end{equation*}
which is estimate \eqref{Condition(otherform)} and therefore (a) and (b) are equivalent. We now show that (b) and (c) are equivalent. Rewriting the term $\int_0^1 \psi_{\alpha, p}(t) \int_0^{t^4} r^{\frac{p-2\alpha-4}{2}}(1-r)^\alpha drdt$ in \eqref{Condition(otherform)} by using integration by parts we get
	\begin{equation*}
		\begin{split}
			&\left[ \int_0^t \psi_{\alpha, p}(r)dr B_{t^4}\left(\frac{p-2\alpha-2}{2}, \alpha +1\right)\right]_0^1 -\int_0^1 \left(\int_0^t(\psi_{\alpha, p}(r)dr\right)4t^3(t^{2(p-2\alpha-4)}(1-t^4)^\alpha)dt\\
			&=B\left(\frac{2+\alpha}{p}, 1 - \frac{2+\alpha}{p}\right)B\left(\frac{p-2\alpha-2}{2}, \alpha+1\right)\\
			&-4\int_0^1 B_t\left(\frac{2+\alpha}{p}, 1 - \frac{2+\alpha}{p}\right)t^{2p-4\alpha-5}(1-t^4)^\alpha dt.
		\end{split}
	\end{equation*}
	Thus, estimate \eqref{Condition(otherform)} %$(\ref{LL})$ 
	takes the form
	\begin{equation*}
		2B\left(\frac{2+\alpha}{p}, 1 -\frac{2+\alpha}{p}\right)\left[\int_0^1 I_t\left(\frac{2+\alpha}{p}, 1 - \frac{2+\alpha}{p}\right)t^{2p-4\alpha-5}(1-t^4)^\alpha dt -\frac{1}{4(\alpha+1)}\right] \le 0.
	\end{equation*}
	Hence (b) is equivalent to \eqref{Condition}.
	%\begin{equation*}
	%	\int_0^1 I_t\left(\frac{2+\alpha}{p}, 1 - \frac{2+\alpha}{p}\right)t^{2p-4\alpha-5}(1-t^4)^\alpha dt - \frac{1}{4(\alpha +1)} \leq 0.
	%\end{equation*}
	\end{proof}
The next two lemmas are the tools needed to prove the main result of the article. They cover the two cases: $2+2\alpha < p < 2+\alpha + \sqrt{\alpha^2+\frac{7}{2}\alpha +3}$ and $2+\alpha +\sqrt{\alpha^2+\frac{7}{2}\alpha +3} \le p <2(2+\alpha)$. We begin with the latter case, because in this case we have obtained a complete result.
\begin{lem}\label{ML1} Let $2 + \alpha + \sqrt{\alpha^2+\frac{7}{2}\alpha+3} \le p < 2(2 + \alpha)$, $\alpha \geq 0$ and $s \in [0,1]$. Then
	\begin{equation*}
		B\left(\frac{2+\alpha}{p}, 1 - \frac{2 + \alpha}{p}\right)H_{\alpha, p}(s) \leq \int_0^1 \psi_{\alpha, p}(t)K_{\alpha,p}(s,t)dt.
	\end{equation*}
%	where $\psi_{\alpha,p}(t) = t^{\frac{2+\alpha}{p}-1}(1-t)^{-\frac{2+\alpha}{p}}$.
\end{lem}
\begin{proof}[Proof.] Let
	\begin{equation*}
		F_{\alpha, p}(s) = B\left(\frac{2+\alpha}{p}, 1 - \frac{2 + \alpha}{p}\right)H_{\alpha, p}(s) - \int_0^1 \psi_{\alpha, p}(t)K_{\alpha,p}(s,t)dt.
	\end{equation*}
 We have
	\begin{equation*}
		F_{\alpha, p}(s) = B\left(\frac{2+\alpha}{p}, 1 - \frac{2 + \alpha}{p}\right)H_{\alpha, p}(s)  - \int_0^s \psi_{\alpha, p}(t)K_{\alpha, p}(s,t)dt + \int_1^s \psi_{\alpha,p}(t)K_{\alpha,p}(s,t)dt
	\end{equation*}
and we will show that $F_{\alpha, p}(s) \leq 0$ for all $s \in [0,1]$.  By a computation  we get
	\begin{equation*}
		F_{\alpha,p}'(s) = 2s^3(1-s^4)^{\alpha}B\left(\frac{2+\alpha}{p}, 1 - \frac{2+\alpha}{p}\right) - 2 G_{\alpha, p}(s)\int_0^s\psi_{\alpha,p}(t)dt,
	\end{equation*}
	where $G_{\alpha, p}(s) = \sum_{k=0}^\infty \binom{\alpha}{k} (-1)^k  s^{2(p-2\alpha -2 + 2k)-1}.$
	Notice that $F'_{\alpha, p }(0) = 0$. By a change of variables, we write
	\begin{equation*}
		B\left(\frac{2+\alpha}{p}, 1-\frac{2+\alpha}{p}\right) = \int_0^s t^{\frac{2+\alpha}{p} -1}(s-t)^{-\frac{2+\alpha}{p}}dt, \hspace{0.10cm} s \neq 0.
	\end{equation*}
Hence, we obtain
	\begin{equation*}
		F'_{\alpha, p}(s) = 2s^3\left[\int_0^s t^{\frac{2+\alpha}{p}-1}\left((1-s^4)^\alpha(s-t)^{-\frac{2+\alpha}{p}} - s^{-3}G_{\alpha, p}(s)(1-t)^{-\frac{2+\alpha}{p}}\right)dt\right].
	\end{equation*}
	Denote $Q_{\alpha,p,s}(t) = (1-s^4)^{\alpha}(s-t)^{-\frac{2+\alpha}{p}} - s^{-3}G_{\alpha, p}(s)(1-t)^{-\frac{2+\alpha}{p}}, t  \leq s$. By solving for zeros of $Q_{\alpha,p,s}$ and observing that  $s^{-3}G_{\alpha, p}(s)= s^{2p-4\alpha-8}(1-s^4)^\alpha$, we have for $s \neq 1$ that
	\begin{equation*}
		\begin{split}
			\frac{1-t}{s-t} &= \left(s^{2p-4\alpha-8}\right)^{\frac{p}{2+\alpha}}.
		\end{split}
	\end{equation*}
By solving for $t$ we get $t = \frac{1-sR_{\alpha, p}(s)}{1-R_{\alpha, p}(s)}$, where 
	$R_{\alpha, p}(s) = \left(s^{2p-4\alpha-8}\right)^{\frac{p}{2+\alpha}}.$
Now zeros exist only if there exists some $t_0$ such that
	\begin{equation*}
		t_0 = \frac{1-sR_{\alpha, p}(s)}{1-R_{\alpha, p}(s)} \geq 0.
	\end{equation*}
	The denominator $1-R_{\alpha, p}(s)$ is strictly less than zero, since $2p - 4\alpha -8 < 0$ and $s \in (0,1)$ so that 
	$s^{2p-4\alpha -8} > 1$. 
For the numerator we have that
	%\begin{equation*}
			$1-sR_{\alpha, p}(s) \leq 0$ if and only if $s^{2p-4\alpha-8+\frac{2+\alpha}{p}} \geq 1$
	%\end{equation*}
	and this inequality holds if and only if $2p -4\alpha -8 + \frac{2+\alpha}{p} \leq 0$. But the last  inequality is not valid (except when $t_0=0$) since
	%\begin{equation*}
	%	\begin{split}
			$2p -4\alpha -8 + \frac{2+\alpha}{p} \leq 0$ if and only if $2p^2 -(4\alpha +8)p +2 +\alpha \leq 0$ which in turn is equivalent to $$2 + \alpha - \sqrt{\alpha^2+\frac{7}{2}\alpha +3} \leq p \leq 2 + \alpha + \sqrt{\alpha^2+\frac{7}{2}\alpha +3}.$$
%		\end{split}
%	\end{equation*}
It follows that $Q_{\alpha,p,s}$ has no zeros or one zero at $t_0 =0$ and since $\lim_{t \to s-} Q_{\alpha,p,s}(t) = +\infty$, we have that $F'_{\alpha, p}(s) \geq 0$ on $(0,1)$. Therefore $F_{\alpha, p}(s)$ is non-decreasing on $(0,1)$. Since $F_{\alpha, p}(1) = 0$,
the statement follows.
\end{proof}
With modifications of Lemma 2.6 in \cite{4} we obtain one generalization of the aforementioned result that works on the weighted Bergman spaces.
\begin{lem}\label{Small p case} Let $2 + 2\alpha < p < 2 + \alpha + \sqrt{\alpha^2+\frac{7}{2}\alpha +3}$, $\alpha \geq 0$ and $s \in [0,1]$. Assume that
	\begin{equation*}
		\int_0^1 I_t\left(\frac{2+\alpha}{p}, 1 - \frac{2+\alpha}{p}\right)t^{2p-4\alpha-5}(1-t^4)^\alpha dt - \frac{1}{4(\alpha +1)} \leq 0.
	\end{equation*}
%	where $\psi_{\alpha,p}(t) = t^{\frac{2+\alpha}{p}-1}(1-t)^{-\frac{2+\alpha}{p}}$. 
	Then
	\begin{equation*}
		B\left(\frac{2+\alpha}{p}, 1 - \frac{2 + \alpha}{p}\right)H_{\alpha, p}(s) \leq \int_0^1 \psi_{\alpha, p}(t)K_{\alpha,p}(s,t)dt.
	\end{equation*}
\end{lem}
\begin{proof}[Proof.]
	Let $F_{\alpha, p}(s)$ be the same as in Lemma \ref{ML1}.  We will show that $F_{\alpha, p}(s) \leq 0$ for all $s \in [0,1]$.
%We have
	%\begin{equation*}
		%F_{\alpha, p}(s) = B\left(\frac{2+\alpha}{p}, 1 - \frac{2 + \alpha}{p}\right)H_{\alpha, p}(s)  - \int_0^s \psi_{\alpha, p}(t)K_{\alpha, p}(s,t)dt + \int_1^s \psi_{\alpha,p}(t)K_{\alpha,p}(s,t)dt
	%\end{equation*}
%and we will show that $F_{\alpha, p}(s) \leq 0$ for all $s \in [0,1]$. By differentiation we get
	%\begin{equation*}
		%\begin{split}
			%F'_{\alpha,p}(s) &= B\left(\frac{2+\alpha}{p}, 1 - \frac{2+\alpha}{p}\right)H'_{\alpha, p}(s) - \frac{d}{ds}[K_{\alpha, p}(s,s)] \int_0^s \psi_{\alpha, p}(t)dt\\
					%	&=  B\left(\frac{2+\alpha}{p}, 1 - \frac{2+\alpha}{p}\right)\left(2s^3(1-s^4)^\alpha\right) - 2\sum_{k= 0}^\infty \binom{\alpha}{k}(-1)^k s^{2(p-2\alpha-2+2k)-1}\int_0^s \psi_{\alpha, p}(t)dt.\\
		%\end{split}
	%\end{equation*}
	Denote again $G_{\alpha, p}(s) = \sum_{k = 0}^\infty \binom{\alpha}{k}(-1)^k s^{2(p-2\alpha-2+2k)-1}$. As  in the proof of Lemma   \ref{ML1} we obtain
	\begin{equation*}
		F'_{\alpha, p}(s) = 2G_{\alpha, p}(s)\left[\frac{s^3(1-s^4)^\alpha}{G_{\alpha, p}(s)}B\left(\frac{2+\alpha}{p}, 1 - \frac{2+\alpha}{p}\right) - \int_0^s \psi_{\alpha,p}(s)dt\right].
	\end{equation*}
	Note that $(1-s^4)^\alpha = \sum_{k = 0}^\infty \binom{\alpha}{k} (-s^4)^{k}$. Using this we get
	\begin{equation*}
		\begin{split}
			F'_{\alpha, p}(s) &= 2s^{2(p-2\alpha-2) -1}\left[\frac{s^3(1-s^4)^\alpha}{s^{2(p-2\alpha -2) -1}}B\left(\frac{2+\alpha}{p},1 - \frac{2+\alpha}{p}\right) - (1-s^4)^\alpha \int_0^s \psi_{\alpha, p}(t)dt\right]\\
						&=2s^{2p-4\alpha-5}(1-s^4)^\alpha\left[s^{8+4\alpha -2p} B\left(\frac{2+\alpha}{p}, 1 - \frac{2+\alpha}{p}\right) - \int_0^s \psi_{\alpha, p}(t)dt\right].
		\end{split}
	\end{equation*}
	Thus we have
	\begin{equation}\label{F}
		F'_{\alpha, p}(s) = 2s^{2p -4\alpha -5}(1-s^4)^\alpha \widetilde{G}_{\alpha, p}(s),
	\end{equation}
	where
	\begin{equation}\label{G1}
		\widetilde{G}_{\alpha, p}(s) = s^{8+4\alpha -2p} B\left(\frac{2+\alpha}{p}, 1 - \frac{2+\alpha}{p}\right) - \int_0^s \psi_{\alpha, p}(t)dt.
	\end{equation}
	By taking the derivative we obtain
	\begin{equation}\label{G2}
		\begin{split}
			\widetilde{G}'_{\alpha, p}(s) &= (8+4\alpha-2p)s^{7+4\alpha -2p}B\left(\frac{2+\alpha}{p}, 1 - \frac{2+\alpha}{p}\right) -\psi_{\alpha, p}(s)\\
								&=\psi_{\alpha, p}(s)(8+4\alpha -2p)B\left(\frac{2+\alpha}{p}, 1 - \frac{2+\alpha}{p}\right)\widetilde{E}_{\alpha, p}(s),
		\end{split}
	\end{equation}
	where
	\begin{equation}\label{E}
		\widetilde{E}_{\alpha, p}(s) = (1-s)^{\frac{2+\alpha}{p}}s^{8+4\alpha-2p - \frac{2+\alpha}{p}}-\frac{1}{(8+4\alpha -2p)B\left(\frac{2+\alpha}{p}, 1 - \frac{2+\alpha}{p}\right)}.
	\end{equation}
	 By differentiating we get
	\begin{equation*}
		\begin{split}
			 \widetilde{E}'_{\alpha, p}(s) &= (8+4\alpha -2p -\frac{2+\alpha}{p})s^{7+4\alpha -2p-\frac{2+\alpha}{p}}(1-s)^{\frac{2+\alpha}{p}}-\frac{2+\alpha}{p}(1-s)^{\frac{2+\alpha}{p}-1}s^{8+4\alpha-2p-\frac{2+\alpha}{p}}\\
								&=s^{7+4\alpha-2p-\frac{2+\alpha}{p}}(1-s)^{\frac{2+\alpha}{p}-1}\left[\left(8+4\alpha-2p-\frac{2+\alpha}{p}\right)(1-s)-\frac{2+\alpha}{p}s\right]\\
								&= s^{7+4\alpha-2p-\frac{2+\alpha}{p}}(1-s)^{\frac{2+\alpha}{p}-1}(8+4\alpha-2p)\left[\frac{8+4\alpha-2p-\frac{2+\alpha}{p}}{8+4\alpha-2p} - s\right].
		\end{split}
	\end{equation*}
	Define
	\begin{equation*}
		p_0 = \frac{8+4\alpha-2p-\frac{2+\alpha}{p}}{8+4\alpha-2p}.
	\end{equation*}
%	Since $2+\alpha < p < 2+\alpha +\sqrt{\alpha^2+\frac{7}{2}\alpha +3}$ if and only if $8+4\alpha -2p -\frac{2+\alpha}{p} > 0$, 
	Since $2+2\alpha < p < 2+\alpha +\sqrt{\alpha^2+\frac{7}{2}\alpha +3}$ implies that $8+4\alpha -2p -\frac{2+\alpha}{p} > 0$, we have that $p_0 \in (0,1)$. Now $\widetilde{E}'_{\alpha, p}(s) \geq 0$ on $[0,p_0]$ and $\widetilde{E}'_{\alpha, p}(s) \leq 0$ on $[p_0,1]$. Therefore $\widetilde{E}_{\alpha, p}$ is non-decreasing on $[0,p_0]$ and non-increasing on $[p_0,1]$. Thus $					\widetilde{E}_{\alpha, p}(p_0) = \max_{0\leq s \leq 1} \widetilde{E}_{\alpha, p}(s)$. Equation (\ref{E}) gives us
	\begin{equation*}
		\widetilde{E}_{\alpha,p}(0) =  \widetilde{E}_{\alpha,p}(1) = -\frac{1}{(8+4\alpha-2p)B\left(\frac{2+\alpha}{p},1-\frac{2+\alpha}{p}\right)} < 0.
	\end{equation*}
	Let us show that $\widetilde{E}_{\alpha, p}(p_0) > 0$. Assume to the contrary that $\widetilde{E}_{\alpha, p}(p_0) \leq 0$. Then $\widetilde{E}_{\alpha,p}(s) \leq 0$ for every $s \in [0,1]$. This implies that $\widetilde{G}'_{\alpha, p}(s) \leq 0$ for all $s \in [0,1]$. Hence $\widetilde{G}_{\alpha, p}$ is non-increasing on $[0,1]$. Since $\widetilde{G}_{\alpha, p}(0) = \widetilde{G}_{\alpha, p}(1) = 0$, we have that $\widetilde{G}_{\alpha, p} \equiv 0$ and therefore $\widetilde{G}'_{\alpha, p} \equiv 0$. Now it follows from \eqref{G2} that $\widetilde{E}_{\alpha, p} \equiv 0$, %and furthermore $\widetilde{E}'_{\alpha, p} \equiv 0$. This 
which contradicts (\ref{E}) and so $\widetilde{E}_{\alpha, p}(p_0) > 0$. There exists $p_1 \in (0,p_0)$ such that $\widetilde{E}_{\alpha, p} < 0$ on $[0,p_1)$, $\widetilde{E}_{\alpha, p}(p_1) = 0$ and $\widetilde{E}_{\alpha, p} > 0$ on $(p_1, p_0]$. Moreover, there exists $p_2 \in (p_0, 1)$ such that $\widetilde{E}_{\alpha, p} > 0$ on $[p_0, p_2)$, $\widetilde{E}_{\alpha, p}(p_2) = 0$ and $\widetilde{E}_{\alpha, p} < 0$ on $(p_2, 1]$. Thus $\widetilde{E}_{\alpha, p} \leq 0$ on $[0, p_1]$, $\widetilde{E}_{\alpha, p} \geq 0$ on $[p_1,p_2]$ and $\widetilde{E}_{\alpha, p} \leq 0$ on $[p_2, 1]$. Utilizing ($\ref{G2})$ we get $\widetilde{G}'_{\alpha, p} \leq 0$ on $[0, p_1]$, $\widetilde{G}'_{\alpha, p} \geq 0$ on $[p_1, p_2]$ and $\widetilde{G}'_{\alpha, p} \leq 0$ on $[p_2, 1]$. Thus $\widetilde{G}_{\alpha, p}$ is non-increasing on $[0, p_1]$, non-decreasing on $[p_1,p_2]$ and non-increasing on $[p_2,1]$. By using (\ref{G1}) we get $\widetilde{G}_{\alpha, p}(0) = \widetilde{G}_{\alpha, p}(1) = 0$. Hence there exists $p_3 \in (p_1, p_2)$ such that $\widetilde{G}_{\alpha, p} \leq 0$ on $[0,p_3]$, $\widetilde{G}_{\alpha, p}(p_3) = 0$ and $\widetilde{G}_{\alpha, p} \geq 0$ on $(p_3, 1]$. Thus $\widetilde{G}_{\alpha, p} \leq 0$ on $[0, p_3]$ and $\widetilde{G}_{\alpha, p} \geq 0$ on $[p_3, 1]$. By (\ref{F}) we get $F_{\alpha, p}' \leq 0$ on $[0, p_3]$ and $F'_{\alpha, p} \geq 0$ on $[p_3, 1]$. Thus $F_{\alpha, p}$ is non-increasing on $[0, p_3]$ and $F_{\alpha, p}$ is non-decreasing on $[p_3, 1]$. If $s \in [0, p_3]$, then $F_{\alpha, p}(s) \leq F_{\alpha, p}(0)$ and if $s \in [p_3, 1]$ then $F_{\alpha, p}(s) \leq F_{\alpha, p}(1) = 0$. It remains to prove that $F_{\alpha, p}(0) \leq 0$. %(in order to prove that $F_{\alpha, p}(s) \leq 0$ for all $s$). 
Now
\begin{equation*}
F_{\alpha,p}(0) = B\left(\frac{2+\alpha}{p}, 1 - \frac{2 + \alpha}{p}\right)H_{\alpha, p}(0) - \int_0^1 \psi_{\alpha, p}(t)K_{\alpha,p}(0,t)dt \le 0
\end{equation*}		
is equivalent to 		
\begin{equation*}		
\int_0^1 I_t\left(\frac{2+\alpha}{p}, 1 - \frac{2+\alpha}{p}\right)t^{2p-4\alpha-5}(1-t^4)^\alpha dt - \frac{1}{4(\alpha +1)} \leq 0
\end{equation*}
					%&=B\left(\frac{2+\alpha}{p}, 1 - \frac{2+\alpha}{p}\right)\left[\frac{1}{2}B\left(\frac{p-2\alpha-2}{2}, \alpha +1\right) - \frac{1}{2(\alpha+1)}\right]\\
					%&-\frac{1}{2}\int_0^1 \psi_{\alpha, p}(t)B_{t^4}\left(\frac{p-2\alpha-2}{2}, \alpha +1\right)dt \leq 0.
					%&= B\left(\frac{2+\alpha}{p}, 1 - \frac{2 + \alpha}{p}\right)\left[\sum_{k = 0}^\infty \binom{\alpha}{k}(-1)^k \frac{1}{p-2\alpha-2+2k} - \frac{1}{2(\alpha +1}\right]\\
					%&-\int_0^1 \psi_{\alpha, p}(t) \left[\sum_{k = 0}^\infty \binom{\alpha}{k}(-1)^k \frac{t^{2(p-2\alpha-2+2k)}}{p-2\alpha-2+2k}\right]dt\\
					%&= B\left(\frac{2+\alpha}{p}, 1 - \frac{2 + \alpha}{p}\right)\left[\sum_{k = 0}^\infty \binom{\alpha}{k}(-1)^k\frac{1}{p-2\alpha-2+2k}-\frac{1}{2(\alpha+1)}\right]\\
					%&-\int_0^1\left(\sum_{k = 0}^\infty \binom{\alpha}{k}(-1)^k\frac{t^{4k}}{p-2\alpha-2+2k}\right)t^{2(p-2\alpha-2)}\psi_{\alpha, p}(t)dt
	%\end{split}
%\end{equation*}
%where the last equality holds by Lemma \ref{Condition1}.
by Lemma \ref{Condition1}, which completes the proof.
\end{proof}

\section{The norm of the Hilbert matrix operator on $A_\alpha^p$}

In this section we provide a proof for the conjecture on the norm of the Hilbert matrix operator on the weighted Bergman spaces $A_{\alpha}^p$ when $2+\alpha +\sqrt{\alpha^2+\frac{7}{2}\alpha+3} \le p < 2(2+\alpha)$ and a sufficient condition for the conjecture to hold when $2+2\alpha < p < 2 + \alpha +\sqrt{\alpha^2+\frac{7}{2}\alpha+3}$. The outline of this section is the following: we begin from the upper estimate (\ref{Normstart2}). In the same way as in \cite{4} a new upper estimate for the right-hand side of (\ref{Normstart2}) is obtained by integrating over an annulus $R_{t^2} = \{z \in \mathbb{C} : t^2 < |z| < 1 \}$. After deriving some further upper estimates we turn to the Taylor series expansion of $(1-r^2)^\alpha$ to be able to proceed further. Finally, we use Lemma \ref{ML1} for large $p$ and Lemma \ref{Small p case} for small $p$ to arrive at an upper estimate for the norm of $\mathcal{H}$.\\
\indent We are now ready to begin. Denote
\begin{equation*}
	\varphi(r) = \frac{1}{\pi}\int_0^{2\pi} |f(re^{it})|^p\, dt %2\left(\frac{1}{2\pi}\int_0^{2\pi} |f(re^{it})|^pdt\right)
\end{equation*}
and $\mathcal{X}(r) = \varphi(r) -\varphi(0)$. If $f$ is analytic in $\mathbb{D}$, it follows that $\varphi$ is non-decreasing and differentiable on the interval $(0,1)$. Thus, $\mathcal{X}$ is also non-decreasing and differentiable on $(0,1)$. Hence,
\begin{equation*}
	\mathcal{X}' \geq 0 \mbox{ on } (0,1)
\end{equation*}
and 
\begin{equation*}
	\mathcal{X}(r) = \int_0^r \mathcal{X}'(s)ds,
\end{equation*}
where $0 < r < 1$. By (\ref{Normstart1}) we have
\begin{equation*}
	\lVert \mathcal{H}(f) \rVert_{A_{\alpha}^p} \leq \int_0^1 \lVert T_t(f)\rVert_{A_{\alpha}^p}dt
\end{equation*}
and by (\ref{Normstart2}) we have an upper estimate for the integrand:
\begin{equation*}
	\lVert T_t(f) \rVert_{A_\alpha^p} \leq   \frac{t^{\frac{2+\alpha}{p} -1}}{(1-t)^{\frac{2+\alpha}{p}}} \left((\alpha+1) \int_{D_t} |w|^{p-2\alpha-4}|f(w)|^p(1-|w|^2)^\alpha dA(w)\right)^{1/p}.
\end{equation*}
We want to work with the expression on the right-hand side of (\ref{Normstart2}). Now by (\ref{Normstart1}) and since $D_t \subset R_{t^2}$ it is enough to show that
\begin{equation}\label{O1}
	\int_0^1 \psi_{\alpha, p}(t)\left(\int_{R_{t^2}} |w|^{p-2\alpha -4}|f(w)|^p dA_\alpha (w)\right)^{1/p}dt \leq \frac{\pi}{\sin{\frac{(2+\alpha)\pi}{p}}}\lVert f \rVert_{A_{\alpha}^p}.
\end{equation}
Now
\begin{equation*}
	\frac{\pi}{\sin{\frac{(2+\alpha)\pi)}{p}}} \lVert f \rVert_{A_{\alpha}^p} = \int_0^1 \psi_{\alpha, p}(t)\left((\alpha+1)\int_0^1 \varphi(r)r(1-r^2)^\alpha dr\right)^{1/p}dt
\end{equation*}
and
\begin{equation*}
	\int_{R_{t^2}} |w|^{p-2\alpha -4} |f(w)|^{p} dA_\alpha(w) = (\alpha+1)\int_{t^2}^1 r^{p-2\alpha -3}\varphi(r) (1-r^2)^\alpha dr.
\end{equation*}
Utilizing this we get that (\ref{O1}) holds if the following inequality is true
\begin{equation}\label{O2}
	\begin{split}
		&\int_0^1 \psi_{\alpha, p}(t)\left(\int_{t^2}^1 r^{p-2\alpha -3} \varphi(r)(1-r^2)^\alpha dr\right)^{1/p}dt \leq\\
		&\int_0^1 \psi_{\alpha, p}(t)\left(\int_0^1 \varphi(r)r(1-r^2)^\alpha dr\right)^{1/p}dt.
	\end{split}
\end{equation}
By the inequality $x^\beta - y^\beta \leq \beta y^{\beta -1}(x-y)$, where $x >0, y >0$ and $\beta \in (0,1)$, we have
\begin{equation*}
	\begin{split}
	&\left(\int_{t^2}^1 r^{p-2\alpha -3} \varphi(r)(1-r^2)^\alpha dr\right)^{1/p} -\left(\int_0^1 \varphi(r)r(1-r^2)^\alpha dr\right)^{1/p}\\
	&\leq \frac{1}{p}\left(\int_0^1 \varphi(r) r(1-r^2)^\alpha dr\right)^{1/p-1}\left(\int_{t^2}^1 r^{p-2\alpha -3} \varphi(r)(1-r^2)^\alpha dr - \int_0^1 \varphi(r)r(1-r^2)^\alpha dr\right).
	\end{split}
\end{equation*}
Therefore, we arrive at the inequality
\begin{equation*}
	\begin{split}
		&\int_0^1 \psi_{\alpha, p}(t)\left[\left(\int_{t^2}^1 r^{p-2\alpha -3} \varphi(r)(1-r^2)^\alpha dr\right)^{1/p} -\left(\int_0^1 \varphi(r)r(1-r^2)^\alpha dr\right)^{1/p}\right]dt\\
		&\leq \frac{1}{p}\left(\int_0^1 \varphi(r) r(1-r^2)^\alpha dr\right)^{1/p -1} \\ 
		&\cdot \int_0^1 \psi_{\alpha, p}(t)\left(\int_{t^2}^1 r^{p-2\alpha-3} \varphi(r)(1-r^2)^\alpha dr - \int_0^1 \varphi(r)r(1-r^2)^\alpha dr\right)dt.
	\end{split}
\end{equation*}
So instead by looking at the expression on the right-hand side, we get that (\ref{O2}) holds if
\begin{equation}
	\int_0^1 \psi_{\alpha, p}(t)\left(\int_{t^2}^1 r^{p-2\alpha-3} \varphi(r)(1-r^2)^\alpha dr - \int_0^1 \varphi(r)r(1-r^2)^\alpha dr\right)dt \leq 0
\end{equation}
or equivalently
\begin{equation}\label{O4}
	\begin{split}
		&\int_0^1 \psi_{\alpha, p}(t)\int_{t^2}^1 r^{p-2\alpha-3} \mathcal{X}(r)(1-r^2)^\alpha drdt\\
 		&+ \varphi(0) \left[\int_0^1 \psi_{\alpha, p}(t) \int_{t^2}^1 r^{p-2\alpha-3}(1-r^2)^\alpha drdt-\int_0^1 \psi_{\alpha, p}(t) \int_0^1 r(1-r^2)^\alpha drdt\right]\\
		&\leq \int_0^1 \psi_{\alpha, p}(t) \int_0^1 \mathcal{X}(r)r(1-r^2)^\alpha drdt.
	\end{split}
\end{equation}
To proceed, we do two things: prove that the first term on the left-hand side is less than or equal to the term on the right-hand side, and prove that the expression in the brackets is less than zero. We begin with the second part. Recall that
\begin{equation*}
	F_{\alpha, p}(s) =  B\left(\frac{2+\alpha}{p}, 1 - \frac{2 + \alpha}{p}\right)H_{\alpha, p}(s) - \int_0^1 \psi_{\alpha, p}(t)K_{\alpha,p}(s,t)dt.
\end{equation*}
Utilizing the definition of $F_{\alpha, p}$ and the dominated convergence theorem, we see that
\begin{equation*}
	\begin{split}
		&\int_0^1 \psi_{\alpha, p}(t) \int_{t^2}^1 r^{p-2\alpha-3}(1-r^2)^\alpha drdt-\int_0^1 \psi_{\alpha, p}(t) \int_0^1 r(1-r^2)^\alpha drdt\\
		&= \int_0^1 \psi_{\alpha, p}(t) \int_{t^2}^1 r^{p-2\alpha-3}\left(\sum_{k =0}^\infty \binom{\alpha}{k}(-1)^k r^{2k}\right) drdt-\int_0^1 \psi_{\alpha, p}(t)\frac{1}{2(\alpha+1)}dt\\
		&=\int_0^1 \psi_{\alpha, p}(t)\left( \sum_{k = 0}^\infty \binom{\alpha}{k}(-1)^k \int_{t^2}^1 r^{p-2\alpha-3+2k}dr-\frac{1}{2(\alpha+1)}\right)dt\\
		&=\int_0^1 \psi_{\alpha, p}(t)\left(\sum_{k = 0}^\infty \binom{\alpha}{k}(-1)^k \frac{1 -t^{2(p-2\alpha-2+2k)}}{p-2\alpha-2+2k}-\frac{1}{2(\alpha+1)}\right)dt\\
		&=\int_0^1 \psi_{\alpha, p}(t)\left(\sum_{k = 0}^\infty \left[\binom{\alpha}{k}(-1)^k \frac{1}{p-2\alpha -2 +2k}\right]-\frac{1}{2(\alpha+1)}\right)dt\\
		&- \int_0^1 \psi_{\alpha, p}(t)\left(\sum_{k = 0}^\infty \binom{\alpha}{k}(-1)^k\frac{t^{2(p-2\alpha-2+2k)}}{p - 2\alpha-2+2k}\right)dt\\
		&=B\left(\frac{2+\alpha}{p}, 1 - \frac{2+\alpha}{p}\right)\left(\sum_{k = 0}^\infty \binom{\alpha}{k}(-1)^k\frac{1}{p-2\alpha-2+2k} - \frac{1}{2(\alpha+1)}\right)\\
		&-\int_0^1 \psi_{\alpha, p}(t)\left(\sum_{k = 0}^\infty \binom{\alpha}{k}(-1)^k \frac{t^{2(p-2\alpha-2+2k)}}{p-2\alpha-2+2k}\right)dt\\
		&= F_{\alpha, p}(0) \leq 0,
	\end{split}
\end{equation*}
which holds according to Lemma \ref{ML1}, if $2 + \alpha + \sqrt{\alpha^2 +\frac{7}{2}\alpha +3} \leq p < 2(2 + \alpha)$, $\alpha \geq 0$. If instead $2+2\alpha < p < 2 +\alpha + \sqrt{\alpha^2+\frac{7}{2}\alpha +3}$ and condition (\ref{Condition}) holds, then Lemma \ref{Small p case} gives the result. For the second part we need to show that
\begin{equation}\label{O5}
	\int_0^1 \psi_{\alpha, p}(t)\int_{t^2}^1 r^{p-2\alpha-3} \mathcal{X}(r)(1-r^2)^\alpha drdt \leq \int_0^1 \psi_{\alpha, p}(t) \int_0^1 \mathcal{X}(r)r(1-r^2)^\alpha drdt.
\end{equation}
By Fubini's theorem we obtain
\begin{equation*}
	\begin{split}
		&\int_{t^2}^1 r^{p-3-2\alpha}\mathcal{X}(r)(1-r^2)^\alpha dr = \int_{t^2}^1 r^{p-3-2\alpha}(1-r^2)^\alpha\int_0^r \mathcal{X}'(s)dsdr\\
		&=\int_0^1 \mathcal{X}'(s)\int_{\max\{s,t^2\}}^1 r^{p-3-2\alpha}(1-r^2)^\alpha drds\\
		&= \int_0^1 \mathcal{X}'(s) \int_{\max\{s,t^2\}}^1 r^{p-3-2\alpha}\sum_{k = 0}^\infty \binom{\alpha}{k}(-1)^k r^{2k}drds.
	\end{split}
\end{equation*}
Now, by the dominated convergence theorem
\begin{equation*}
	\begin{split}
		&\int_0^1 \mathcal{X}'(s)\left(\sum_{k = 0}^\infty \binom{\alpha}{k}(-1)^k\int_{\max\{s,t^2\}}^1r^{p-3-2\alpha+2k}dr\right)ds\\
		&=\int_0^1 \mathcal{X}'(s)\left( \sum_{k = 0}^\infty \frac{\binom{\alpha}{k}(-1)^k}{p-2-2\alpha+2k}\left(1 - \max\{s,t^2\}^{p-2-2\alpha+2k}\right)\right)ds
	\end{split}
\end{equation*}
and
\begin{equation*}
	\begin{split}
		\int_0^1 \mathcal{X}(r)r(1-r^2)^\alpha dr &= \int_0^1 \int_0^r \mathcal{X}'(s)ds\, r(1-r^2)^\alpha dr\\
									&=\int_0^1 \mathcal{X}'(s)\int_s^1 r(1-r^2)^\alpha drds\\
									&=\int_0^1 \mathcal{X}'(s)\left[\frac{1}{2(\alpha+1)}(1-s^2)^{\alpha+1}\right]ds.
	\end{split}
\end{equation*}
So inequality (\ref{O5}) takes the form
\begin{equation*}
	\begin{split}
		&\int_0^1\psi_{\alpha, p}(t)\int_0^1 \mathcal{X}'(s)\left( \sum_{k = 0}^\infty \frac{\binom{\alpha}{k}(-1)^k}{p-2-2\alpha+2k}(1 - \max\{s,t^2\}^{p-2-2\alpha+2k})\right)dsdt\\
		&\leq \int_0^1 \psi_{\alpha, p}(t) \int_0^1 \mathcal{X}'(s)\left[\frac{1}{2(\alpha+1)}(1-s^2)^{\alpha+1}\right]dsdt
	\end{split}
\end{equation*}
or equivalently
\begin{equation*}
	\begin{split}
		&\int_0^1 \mathcal{X}'(s)\Bigg[\sum_{k = 0}^\infty \frac{\binom{\alpha}{k}(-1)^k}{p-2-2\alpha+2k}\int_0^1\psi_{\alpha,p}(t)dt -\left(\frac{1}{2(\alpha+1)}(1-s^2)^{\alpha+1}\right)\int_0^1 \psi_{\alpha, p}(t)dt\\
		&-\int_0^1 \psi_{\alpha, p}(t) \sum_{k = 0}^\infty \frac{\binom{\alpha}{k}(-1)^k}{p-2-2\alpha+2k} \max\{s,t^2\}^{p-2-2\alpha +2k}dt\Bigg]ds \leq 0,
	\end{split}
\end{equation*}
which is the same as
\begin{equation*}
	\begin{split}
		&\int_0^1 \mathcal{X}'(s)\Bigg[\left(\sum_{k = 0}^\infty \frac{\binom{\alpha}{k}(-1)^k}{p-2-2\alpha+2k} - \frac{1}{2(\alpha+1)}(1-s^2)^{\alpha+1}\right)B\left(\frac{2+\alpha}{p}, 1 - \frac{2+\alpha}{p}\right)\\
		&-\int_0^1 \psi_{\alpha,p}(t) \sum_{k = 0}^\infty \frac{\binom{\alpha}{k}(-1)^k}{p-2-2\alpha+2k}\max\{s,t^2\}^{p-2-2\alpha+2k}dt\Bigg]ds \leq 0.
	\end{split}
\end{equation*}
Note that $\sum_{k = 0}^\infty \frac{\binom{\alpha}{k}(-1)^k}{p-2-2\alpha+2k}\max\{s,t^2\}^{p-2-2\alpha+2k} = K_{\alpha, p}(\sqrt{s},t)$.
Hence, we have to show that
\begin{equation*}
	\begin{split}
		&\int_0^1 \mathcal{X}'(s)\Bigg[\left(\sum_{k = 0}^\infty \frac{\binom{\alpha}{k}(-1)^k}{p-2-2\alpha+2k} - \frac{1}{2(\alpha+1)}(1-s^2)^{\alpha+1}\right)B\left(\frac{2+\alpha}{p}, 1 - \frac{2+\alpha}{p}\right)\\
		&-\int_0^1 \psi_{\alpha,p}(t)K_{\alpha,p}(\sqrt{s},t)dt \Bigg]ds \leq 0.
	\end{split}
\end{equation*}
Observe that $H_{\alpha, p}(\sqrt{s}) = \sum_{k = 0}^\infty \frac{\binom{\alpha}{k}(-1)^k}{p-2-2\alpha+2k} - \frac{1}{2(\alpha+1)}(1-s^2)^{\alpha+1}$, so we get
\begin{equation*}
	\begin{split}
		&\int_0^1 \mathcal{X}'(s)\left[H_{\alpha, p}(\sqrt{s})B\left(\frac{2+\alpha}{p}, 1 -\frac{2+\alpha}{p}\right) - \int_0^1\psi_{\alpha, p}(t)K_{\alpha, p}(\sqrt{s}, t)dt\right]ds\\
		&= \int_0^1 \mathcal{X}'(s)F_{\alpha, p}(\sqrt{s})ds \leq 0,
	\end{split}
\end{equation*}
which holds by Lemma \ref{ML1} when $2+\alpha +\sqrt{\alpha^2+\frac{7}{2}\alpha+3} \leq p < 2(2+\alpha)$ and by Lemma \ref{Small p case} when $2+2\alpha < p < 2+\alpha + \sqrt{\alpha^2+\frac{7}{2}\alpha +3}$ and given that condition \eqref{Condition} holds. 
Therefore, when $2 + \alpha + \sqrt{\alpha^2 + \frac{7}{2}\alpha +3} \leq p < 2(2 + \alpha)$, we have
\begin{equation*}
	\lVert \mathcal{H} \rVert_{A_{\alpha}^p \to A_{\alpha}^p} \leq \frac{\pi}{\sin\frac{(2+\alpha)\pi}{p}}.
\end{equation*}
In the case of $2+2\alpha < p < 2+\alpha+\sqrt{\alpha^2+\frac{7}{2}\alpha +3}$ it holds that
\begin{equation*}
	\lVert \mathcal{H} \rVert_{A_{\alpha}^p \to A_{\alpha}^p} \leq \frac{\pi}{\sin\frac{(2+\alpha)\pi}{p}},
\end{equation*}
if $\alpha$ is such that condition (\ref{Condition}) holds. We are now ready to proceed to the proof of the main theorem.
\\
\begin{proof}[\textbf{Proof of Theorem \ref{Main Theorem}}]
	Note that the lower bound of the norm of $\mathcal{H}$ holds for all $\alpha \geq 0$ and all $2+2\alpha < p <2(2+\alpha)$ by Theorem 1.1 in \cite{6}. For the upper bound of the norm of $\mathcal{H}$ we have by the above argument 
that if (a) is true, then the conclusion of the theorem holds. If instead (b) is assumed, then by Lemma \ref{Small p case} we have that $F_{\alpha, p}(s)\le 0$ for all $s \in [0,1]$, %by Lemma \ref{Condition1} condition (\ref{Condition}) is equivalent to $F_{\alpha, p}(0) \leq 0$, 
where $F_{\alpha, p}$ is the function defined in the proof of Lemma \ref{ML1}. %This is precisely what is needed for Lemma \ref{Small p case} to hold. 
Utilizing this in combination with the argument on the previous pages we then obtain the upper bound of the norm of $\mathcal{H}$.
\end{proof}
The next lemma gives us a stronger condition than condition (\ref{Condition}). This new condition is useful for our purposes.
Note that when $\alpha = 0$ the inequality in the lemma becomes
\begin{equation*}
	\frac{1}{2p-4}-\frac{1}{(2p-4)^2}\frac{1}{B\left(\frac{2}{p}, 2p-4\right)} -\frac{1}{4} \leq 0,
\end{equation*}
which is equivalent to
\begin{equation*}
	B\left(\frac{2}{p}, 2p-4\right) \leq \frac{1}{(p-2)(4-p)}.
\end{equation*}
The inequality $B\left(\frac{2}{p}, 2p-4\right) \leq \frac{1}{(p-2)(4-p)}$ holds for $2 < p < 4$, by Lemma 2.5 in \cite{4} or Lemma 3.2 in \cite{8}. The above inequality is one ingredient in the proof of the main result in \cite{4}.
\begin{lem}\label{WC2} Let $\alpha \in [0,1]$ or $\alpha \in [2,3]$ and $2+2\alpha < p < 2(2+\alpha)$. If
	\begin{equation*}
		\begin{split}
			&\frac{1}{2p-4\alpha-4}-\frac{1}{(2p-4\alpha-4)^2}\frac{1}{B\left(\frac{2+\alpha}{p},2p-4\alpha-4\right)}\\
			&-\alpha\left(\frac{1}{2p-4\alpha}-\frac{1}{(2p-4\alpha)^2}\frac{1}{B\left(\frac{2+\alpha}{p},2p-4\alpha\right)}\right)\\
			&+\frac{\alpha(\alpha-1)}{2}\left(\frac{1}{2p-4\alpha+4}-\frac{1}{(2p-4\alpha+4)^2}\frac{1}{B\left(\frac{2+\alpha}{p},2p-4\alpha+4\right)}\right) -\frac{1}{4(\alpha+1)} \leq 0,
		\end{split}
	\end{equation*}
	then
	\begin{equation*}
		\int_0^1 I_t\left(\frac{2+\alpha}{p}, 1 - \frac{2+\alpha}{p}\right)t^{2p-4\alpha-5}(1-t^4)^\alpha dt - \frac{1}{4(\alpha +1)} \leq 0.
	\end{equation*}
\end{lem}
\begin{proof}
	We begin with
	\begin{equation*}
		\begin{split}
			&\int_0^1 I_t\left(\frac{2+\alpha}{p}, 1 - \frac{2+\alpha}{p}\right)t^{2p-4\alpha-5}(1-t^4)^\alpha dt - \frac{1}{4(\alpha +1)}\\
			&=\frac{1}{B\left(\frac{2+\alpha}{p}, 1-\frac{2+\alpha}{p}\right)}\left(\int_0^1 B_t\left(\frac{2+\alpha}{p}, 1 - \frac{2+\alpha}{p}\right)t^{2p-4\alpha-5}(1-t^4)^\alpha dt - \frac{B\left(\frac{2+\alpha}{p}, 1-\frac{2+\alpha}{p}\right)}{4(\alpha +1)}\right).
		\end{split}
	\end{equation*}
	Working with the expression in the parenthesis and using $(1-t^4)^\alpha \leq 1 - \alpha t^4+\frac{\alpha(\alpha -1)}{2}t^8$, see \cite[p. 531]{4}, we get
	\begin{equation*}
		\begin{split}
			&\int_0^1 B_t\left(\frac{2+\alpha}{p}, 1 - \frac{2+\alpha}{p}\right)t^{2p-4\alpha-5}(1-t^4)^\alpha dt - \frac{B\left(\frac{2+\alpha}{p}, 1-\frac{2+\alpha}{p}\right)}{4(\alpha +1)}\\
			&\leq \int_0^1 B_t\left(\frac{2+\alpha}{p}, 1 - \frac{2+\alpha}{p}\right)t^{2p-4\alpha-5}dt -\alpha \int_0^1 B_t\left(\frac{2+\alpha}{p}, 1-\frac{2+\alpha}{p}\right)t^{2p-4\alpha-1}dt\\ 
			&+\frac{\alpha(\alpha-1)}{2}\int_0^1 B_t\left(\frac{2+\alpha}{p}, 1 -\frac{2+\alpha}{p}\right)t^{2p-4\alpha+3}dt-\frac{B\left(\frac{2+\alpha}{p}, 1-\frac{2+\alpha}{p}\right)}{4(\alpha +1)}.
		\end{split}
	\end{equation*}
	Denote
	\begin{equation*}
		\begin{split}
			&J_1 = \int_0^1 B_t\left(\frac{2+\alpha}{p}, 1 - \frac{2+\alpha}{p}\right)t^{2p-4\alpha-5}dt;\\
			&J_2 = \int_0^1 B_t\left(\frac{2+\alpha}{p}, 1-\frac{2+\alpha}{p}\right)t^{2p-4\alpha-1}dt;\\
			&J_3 = \int_0^1 B_t\left(\frac{2+\alpha}{p}, 1 -\frac{2+\alpha}{p}\right)t^{2p-4\alpha+3}dt.
		\end{split}
	\end{equation*}
	Using integration by parts we get
	\begin{equation*}
		\begin{split}
			J_1 &= \int_0^1 \left(\int_0^t s^{\frac{2+\alpha}{p}-1}(1-s)^{-\frac{2+\alpha}{p}}ds\right) t^{2p-4\alpha -5}dt\\
				&=\left[\frac{1}{2p-4\alpha-4}t^{2p-4\alpha-4}\int_0^t s^{\frac{2+\alpha}{p}-1}(1-s)^{-\frac{2+\alpha}{p}}ds\right]_0^1\\
				&-\int_0^1 \frac{1}{2p-4\alpha-4}t^{2p-4\alpha-5+\frac{2+\alpha}{p}}(1-t)^{-\frac{2+\alpha}{p}}dt\\
				&=\frac{1}{2p-4\alpha-4}\left[B\left(\frac{2+\alpha}{p}, 1-\frac{2+\alpha}{p}\right) - B\left(2p-4\alpha-4+\frac{2+\alpha}{p}, 1 -\frac{2+\alpha}{p}\right)\right].
		\end{split}
	\end{equation*}
	For $J_2$ and $J_3$ we similarly get
	\begin{equation*}
		\begin{split}
			&J_2 = \frac{1}{2p-4\alpha}\left[B\left(\frac{2+\alpha}{p}, 1 -\frac{2+\alpha}{p}\right) - B\left(2p-4\alpha+\frac{2+\alpha}{p}, 1-\frac{2+\alpha}{p}\right)\right];\\
			&J_3 = \frac{1}{2p-4\alpha+4}\left[B\left(\frac{2+\alpha}{p}, 1 -\frac{2+\alpha}{p}\right) - B\left(2p-4\alpha+4+\frac{2+\alpha}{p}, 1-\frac{2+\alpha}{p}\right)\right].
		\end{split}
	\end{equation*}
	By combining these we arrive at 
%	\begin{equation*}
%		\begin{split}
%			&\frac{1}{2p-4\alpha-4}\left[B\left(\frac{2+\alpha}{p}, 1-\frac{2+\alpha}{p}\right) - B\left(2p-4\alpha-4+\frac{2+\alpha}{p}, 1 -\frac{2+\alpha}{p}\right)\right]\\
%			&-\frac{\alpha}{2p-4\alpha}\left[B\left(\frac{2+\alpha}{p}, 1 -\frac{2+\alpha}{p}\right) - B\left(2p-4\alpha+\frac{2+\alpha}{p}, 1-\frac{2+\alpha}{p}\right)\right]\\
%			&+\frac{\alpha(\alpha-1)}{2(2p-4\alpha+4)}\left[B\left(\frac{2+\alpha}{p}, 1 -\frac{2+\alpha}{p}\right) - B\left(2p-4\alpha+4+\frac{2+\alpha}{p}, 1-\frac{2+\alpha}{p}\right)\right]\\
%			&-\frac{B\left(\frac{2+\alpha}{p}, 1-\frac{2+\alpha}{p}\right)}{4(\alpha +1)}.
%		\end{split}
%	\end{equation*}
	%The above expression can further be written as
	\begin{equation*}
		\begin{split}
	&J_1-\alpha J_2+\frac{\alpha(\alpha-1)}{2}J_3-\frac{B\left(\frac{2+\alpha}{p}, 1-\frac{2+\alpha}{p}\right)}{4(\alpha+1)} \\
	&= B\left(\frac{2+\alpha}{p}, 1-\frac{2+\alpha}{p}\right)\Bigg[\frac{1}{2p-4\alpha-4}\left(1-\frac{B\left(2p-4\alpha-4+\frac{2+\alpha}{p}, 1-\frac{2+\alpha}{p}\right)}{B\left(\frac{2+\alpha}{p}, 1-\frac{2+\alpha}{p}\right)}\right)\\
											&-\frac{\alpha}{2p-4\alpha}\left(1-\frac{B\left(2p-4\alpha+\frac{2+\alpha}{p}, 1-\frac{2+\alpha}{p}\right)}{B\left(\frac{2+\alpha}{p}, 1-\frac{2+\alpha}{p}\right)}\right) \\
											&+\frac{\alpha(\alpha-1)}{2(2p-4\alpha+4)}\left(1-\frac{B\left(2p-4\alpha+4+\frac{2+\alpha}{p}, 1-\frac{2+\alpha}{p}\right)}{B\left(\frac{2+\alpha}{p}, 1-\frac{2+\alpha}{p}\right)}\right)
											-\frac{1}{4(\alpha +1)}\Bigg].
		\end{split}
	\end{equation*}
	Now finally we rewrite the Beta functions occurring in the above expression by using the formulas $\Gamma(z+1) = z\Gamma(z)$ and $B(x,y) = \frac{\Gamma(x)\Gamma(y)}{\Gamma(x+y)}$ and arrive at the inequality %\frac{B\left(2p-4\alpha-4+\frac{2+\alpha}{p}, 1-\frac{2+\alpha}{p}\right)}{B\left(\frac{2+\alpha}{p},1-\frac{2+\alpha}{p}\right)} &=\Gamma(\frac{2+\alpha}{p}
\begin{equation*}
	\begin{split}
		&\frac{1}{2p-4\alpha-4}-\frac{1}{(2p-4\alpha-4)^2}\frac{1}{B\left(\frac{2+\alpha}{p},2p-4\alpha-4\right)}\\
		&-\alpha\left(\frac{1}{2p-4\alpha}-\frac{1}{(2p-4\alpha)^2}\frac{1}{B\left(\frac{2+\alpha}{p},2p-4\alpha\right)}\right)\\
		&+\frac{\alpha(\alpha-1)}{2}\left(\frac{1}{2p-4\alpha+4}-\frac{1}{(2p-4\alpha+4)^2}\frac{1}{B\left(\frac{2+\alpha}{p},2p-4\alpha+4\right)}\right) -\frac{1}{4(\alpha+1)} \leq 0,
	\end{split}
\end{equation*}
which completes the proof.
\end{proof}
We now turn our attention to small $\alpha \geq 0$. The following result shows that condition (\ref{Condition}) holds for some values of the parameter $p$.

\begin{prop} \ 
	\begin{enumerate}

		\item[(a)] If $0 < \alpha \leq \frac{1}{19}$ and $2+2\alpha < p < \frac{5}{2} + 2\alpha$, then condition (\ref{Condition}) holds;

		\item[(b)] If $0 < \alpha \leq \frac{1}{15000}$ and $\frac{5}{2}+2\alpha \le p \le 3+2\alpha$, then condition (\ref{Condition}) holds.
	\end{enumerate}
\end{prop}
\begin{proof}
	We will begin with (a). By Lemma \ref{WC2} we want to show that
	\begin{equation}\label{small p condition}
		\begin{split} S(\alpha, p):= 
			&\frac{1}{2p-4\alpha-4}-\frac{1}{(2p-4\alpha-4)^2}\frac{1}{B\left(\frac{2+\alpha}{p},2p-4\alpha-4\right)}\\
			&-\alpha\left(\frac{1}{2p-4\alpha}-\frac{1}{(2p-4\alpha)^2}\frac{1}{B\left(\frac{2+\alpha}{p},2p-4\alpha\right)}\right)\\
			&+\frac{\alpha(\alpha-1)}{2}\left(\frac{1}{2p-4\alpha+4}-\frac{1}{(2p-4\alpha+4)^2}\frac{1}{B\left(\frac{2+\alpha}{p},2p-4\alpha+4\right)}\right)\\
			&-\frac{1}{4(\alpha+1)} \leq 0.
		\end{split}
	\end{equation}
	Since $\alpha \geq 0$ and $2+2\alpha < p < \frac{5}{2} + 2\alpha$, we have by Lemma \ref{11}
	\begin{equation*}
		\begin{split}
						 %\leq &\frac{1}{2p-4\alpha-4}\left[1-\frac{1}{(2p-4\alpha-4)\frac{1}{(2p-4\alpha-4)\frac{2+\alpha}{p}}\cdot \frac{2p-4\alpha-4+\frac{2+\alpha}{p}}{1+(2p-4\alpha-4)\frac{2+\alpha}{p}}}\right]\\
						%&-\frac{\alpha}{2p-4\alpha}\left[1-\frac{1}{(2p-4\alpha)\frac{1}{(2p-4\alpha)\frac{2+\alpha}{p}}\cdot \frac{2p-4\alpha+\frac{2+\alpha}{p}}{1+(2p-4\alpha)\frac{2+\alpha}{p}}}\right]\\
						%&+\frac{\alpha(\alpha-1)}{2(2p-4\alpha+4)}\left[1-\frac{1}{(2p-4\alpha+4)\frac{1}{(2p-4\alpha+4)\frac{2+\alpha}{p}}\cdot \frac{2p-4\alpha+4+\frac{2+\alpha}{p}}{1+(2p-4\alpha+4)\frac{2+\alpha}{p}}}\right]\\
						%&-\frac{1}{4(\alpha +1)}.
			S(\alpha,p) &\leq \frac{1}{2p-4\alpha-4}\left[1-\frac{(2+\alpha)}{p}\frac{(1+(2p-4\alpha-4)\frac{2+\alpha}{p})}{2p-4\alpha-4+\frac{2+\alpha}{p}}\right]\\
					&-\frac{\alpha}{2p-4\alpha}\left[1-\frac{(2+\alpha)}{p}\frac{(1+(2p-4\alpha)\frac{2+\alpha}{p})}{2p-4\alpha+\frac{2+\alpha}{p}}\right]\\
					&+\frac{\alpha(\alpha-1)}{2(2p-4\alpha+4)}\left[1-\frac{(2+\alpha)}{p}\frac{(1+(2p-4\alpha+4)\frac{2+\alpha}{p})}{2p-4\alpha+4+\frac{2+\alpha}{p}}\right] -\frac{1}{4(\alpha+1)}.
		\end{split}
	\end{equation*}
	By simplifying the above expression we get the following estimate
	\begin{equation*}
		S(\alpha, p) \leq \frac{g_0(p) + h(\alpha, p)}{u(\alpha, p)},
	\end{equation*}
	where
	\begin{equation*}
		\begin{split}
			g_0(p) &= 64 + 136 p + 112 p^2 + 88 p^3 + 32 p^4 - 40 p^5 - 16 p^6 + 8 p^7;\\
			u(\alpha,p) &= 4 (1 + \alpha) p (-2 - \alpha + 4 \alpha p - 2 p^2) (-2 - \alpha - 4 p + 4 \alpha p -2 p^2) (-2 - \alpha + 4 p + 4 \alpha p - 2 p^2)
		\end{split}
	\end{equation*}
	and
	\begin{equation*}
			\begin{split}
				h(\alpha, p) &= \alpha^7 (2 - 16 p + 32 p^2) + \alpha^6 (12 - 72 p + 104 p^2 - 32 p^3)\\
						&+ \alpha^5 (14 - 32 p - 2 p^2 - 64 p^3 - 24 p^4)\\
						&+ \alpha^4 (-44 + 248 p - 236 p^2 + 104 p^3 + 40 p^4 + 32 p^5)\\
						 &+ \alpha^3 (-112 + 353 p - 162 p^2 + 288 p^3 - 72 p^4 - 48 p^5 - 8 p^6)\\
						&+ \alpha^2 (-32 + 134 p - 28 p^2 + 174 p^3 - 120 p^4 + 64 p^5 + 16 p^6)\\
						 &+ \alpha (96 + 140 p + 72 p^2 + 72 p^3 +28 p^5 - 40 p^6).
			\end{split}
	\end{equation*}
	Observing that $g_0(p) = 8 (1 + p)^2 (1 + p^2) (8 + p - 4 p^2 + p^3)$ and $p^3-4p^2+p+8 > 1$ we have for $2+2\alpha < p < \frac{5}{2}+2\alpha$ that
	\begin{equation*}
		g_0(p) \geq  8 (1 + p)^2 (1 + p^2) > 8(1+2)^2(1+2^2) = 360.
	\end{equation*}
	It is also easily seen that $u(\alpha, p) < 0$ on the interval $2+2\alpha < p < \frac{5}{2}+2\alpha$. We now turn our attention to the $h(\alpha, p)$ term. Our aim is to show that $h(\alpha, p)$ is non-increasing with respect to both its parameters. We begin by showing it is non-increasing with respect to the parameter $p$.
	\begin{equation*}
		\begin{split}
			h'(\alpha, p) &= \alpha^7 (-16 + 64 p) + \alpha^6 (-72 + 208 p - 96 p^2)\\
					&+ \alpha^5 (-32 - 4 p - 192 p^2 - 96 p^3)\\
					&+ \alpha^4 (248 - 472 p + 312 p^2 + 160 p^3 + 160 p^4)\\
					&+ \alpha^3 (353 - 324 p + 864 p^2 - 288 p^3 - 240 p^4 - 48 p^5)\\
					&+ \alpha^2 (134 - 56 p + 522 p^2 - 480 p^3 + 320 p^4 + 96 p^5)\\
					&+ \alpha (140 + 144 p + 216 p^2 + 140 p^4 - 240 p^5).
		\end{split}
	\end{equation*}
	Proceeding like this and differentiating four times we arrive at
	\begin{equation*}
		\begin{split}
			h^{(4)}(\alpha, p) &= -2880 \alpha^2 - 1728 \alpha^3 + 960 \alpha^4 - 576 \alpha^5 + 3360 \alpha p + 7680 \alpha^2 p\\
						&-5760 \alpha^3 p + 3840 \alpha^4 p - 14400 \alpha p^2 + 5760 \alpha^2 p^2 - 2880 \alpha^3 p^2\\
						&< -2880 \alpha^2 - 1728 \alpha^3 + 960 \alpha^4 - 576 \alpha^5\\
						&+(3360\alpha +7680\alpha^2+3840\alpha^4-14400\alpha+5760\alpha^2-2880\alpha^3)p^2\\
						&=-2880\alpha^2-1728\alpha^3+960\alpha^4-576\alpha^5\\
						&+(-11040\alpha +13440\alpha^2-2880\alpha^3+3840\alpha^4)p^2 < 0,
		\end{split}
	\end{equation*}
	since the coefficient of $p^2$ inside the parentheses is negative for small $\alpha$. %since the last parenthesis is negative for small $\alpha$. 
	Hence, $h^{'''}(\alpha, p)$ is non-increasing and we get
	\begin{equation*}
		\begin{split}
			h'''(\alpha, p) &< h'''(\alpha, 2+2\alpha) = -12 \alpha (2604 + 6313 \alpha + 4864 \alpha^2 + 1356 \alpha^3 + 288 \alpha^4 + 112 \alpha^5)\\
					&<0. %&= 432 \alpha + 1044 \alpha^2 + 1728 \alpha^3 + 624 \alpha^4 - 384 \alpha^5 - 192 \alpha^6\\
 					%&-2880 a^2 p - 1728 a^3 p + 960 a^4 p - 576 a^5 p + 1680 a p^2\\
					%&+3840 a^2 p^2 - 2880 a^3 p^2 + 1920 a^4 p^2 - 4800 a p^3\\
					%&+1920 a^2 p^3 - 960 a^3 p^3 
		\end{split}
	\end{equation*}
	Thus, $h''(\alpha, p)$ is non-increasing and further
	\begin{equation*}
		\begin{split}
			h''(\alpha, p) &< h''(\alpha, 2+2\alpha) = -4 \alpha (3428 + 12076 \alpha + 15399 \alpha^2 + 8830 \alpha^3\\
					&+ 2313 \alpha^4 + 332 \alpha^5 + 48 \alpha^6)\\
					&<0.
		\end{split}
	\end{equation*}
	Therefore, $h'(\alpha, p)$ is non-increasing so
	\begin{equation*}
		\begin{split}
			h'(\alpha, p) &< h'(\alpha, 2+2\alpha) = -\alpha (4148 + 20962 \alpha + 38631 \alpha^2 + 35128 \alpha^3\\
					&+ 16600 \alpha^4 + 3664 \alpha^5 + 240 \alpha^6)\\
					&< 0.
		\end{split}
	\end{equation*}
	Hence, $h(\alpha, p)$ is non-increasing and
	\begin{equation*}
		\begin{split}
		h(\alpha, p) \geq h(\alpha, \frac{5}{2}+2\alpha) &= -\frac{1}{8}\alpha (40082 + 191376 \alpha + 349873 \alpha^2 + 341132 \alpha^3\\
										& + 197168 \alpha^4 + 65648 \alpha^5 + 9776 \alpha^6).
		\end{split}
	\end{equation*}
	From the above we then see that $h(\alpha, \frac{5}{2}+2\alpha)$ is non-increasing and so
	\begin{equation*}
		h(\alpha, p)  \geq h(\alpha, \frac{5}{2}+2\alpha) \geq h(\frac{1}{19}, \frac{5}{2}+\frac{2}{19}) \approx -336.677,
	\end{equation*}
	for all $2+2\alpha < p < \frac{5}{2}+2\alpha$ and $0 < \alpha \leq \frac{1}{19}$. Recall that
	\begin{equation*}
		S(\alpha, p) \leq \frac{g_0(p) + h(\alpha, p)}{u(\alpha,p)},
	\end{equation*}
	where $g_0(p) > 360$ and $u(\alpha, p) < 0$ on $2+2\alpha < p < \frac{5}{2}+2\alpha$. Together with $h(\alpha, p) \geq -336.677$ we get
	\begin{equation*}
		S(\alpha, p) < 0
	\end{equation*}
	and so condition (\ref{Condition}) holds by Lemma \ref{WC2}. Part (b) can be proved in a similar way as part (a). In \eqref{small p condition}, one can estimate from above the term $(1-t)^{2p-4\alpha-5}$ by its second-degree Taylor polynomial when $2p-4\alpha-5 \in [0,1]$, i.e., $5/2+2\alpha \le p \le 3+2\alpha$. We get the following upper estimate
	\begin{equation*}
		S(\alpha, p) \leq \frac{f(p) + h(\alpha, p)}{u(\alpha, p)},
	\end{equation*}
	where $u(\alpha, p) < 0$, $h(\alpha, p) \leq 453.23$ and $f(p) < -473.67$. The only difficult part is showing that $f(p) < -473.67$; one way of doing this is by using Sturm's theorem, see \cite{9}.
\end{proof}

Now we concentrate on the $\alpha = 1$ case and so we investigate $p$ when $4 < p < 3+\sqrt{7.5}$. The $\alpha = 1$ case is of interest since in this case the two expressions in Lemma \ref{WC2} are identical, and therefore this case is an example of when condition (\ref{Condition}) does not hold for all $2+2\alpha < p < 2+\alpha +\sqrt{\alpha^2+\frac{7}{2}\alpha +3}$. Note that by Lemma \ref{Condition1} the expression in Lemma \ref{WC2} is the same as 
$$\frac{1}{2B\left(\frac{3}{p},1-\frac{3}{p}\right)}F_{1,p}(0),$$ 
where $F_{1,p}(0)$ is defined in Lemma \ref{ML1}. Hence, by using properties of the Gamma function, we may write:
\begin{equation*}
	\begin{split}
		F_{1,p}(0) =&\frac{1}{2}B\left(\frac{3}{p}, 1-\frac{3}{p}\right)\Bigg[\frac{2}{p-4} - \frac{2}{p-2} -\frac{1}{2} -\frac{1}{(p-4)^2 B\left(\frac{3}{p},2p-8\right)}\\
		&+ \frac{1}{B\left(\frac{3}{p},2p-8\right)}\frac{(2p-5+3/p)(2p-6+3/p)(2p-7+3/p)(2p-8+3/p)}{(p-2)^2(2p-5)(2p-6)(2p-7)(2p-8)}\Bigg].
	\end{split}
\end{equation*}
The expression in the brackets can be written
\begin{equation*}
	\begin{split}
		&\frac{2}{p-4} - \frac{2}{p-2} -\frac{1}{2} -\frac{1}{(p-4)^2B\left(\frac{3}{p}, 2p-8\right)}\\
		&+ \frac{1}{B\left(\frac{3}{p}, 2p-8\right)}\frac{(2p-5+3/p)(2p-6+3/p)(2p-7+3/p)(2p-8+3/p)}{(p-2)^2(2p-5)(2p-6)(2p-7)(2p-8)}\\
		&=\frac{1}{B\left(\frac{3}{p}, 2p-8\right)}\Bigg[\left(\frac{2}{p-4}-\frac{2}{p-2}-\frac{1}{2}\right)B\left(\frac{3}{p}, 2p-8\right) -\frac{1}{(p-4)^2}\\
		&+\frac{(2p-5+3/p)(2p-6+3/p)(2p-7+3/p)(2p-8+3/p)}{(p-2)^2(2p-5)(2p-6)(2p-7)(2p-8)}\Bigg].
	\end{split}
\end{equation*}
So in the $\alpha = 1$ case the inequality in Lemma \ref{WC2} becomes
\begin{equation*}
	\begin{split}
		&\left(\frac{2}{p-4}-\frac{2}{p-2}-\frac{1}{2}\right)B\left(\frac{3}{p}, 2p-8\right) -\frac{1}{(p-4)^2}\\
		&+\frac{(2p-5+3/p)(2p-6+3/p)(2p-7+3/p)(2p-8+3/p)}{(p-2)^2(2p-5)(2p-6)(2p-7)(2p-8)} \leq 0.
	\end{split}
\end{equation*}
Now by multiplying with the positive term $(p-4)(p-2)$, we arrive at
\begin{equation}\label{Exampleeq}
	\begin{split}
		&\frac{1}{2}(-p^2+6p)B\left(\frac{3}{p}, 2p-8\right) -\frac{p-2}{p-4}\\
		&+\frac{(2p-5+3/p)(2p-6+3/p)(2p-7+3/p)(2p-8+3/p)}{2(p-2)(2p-5)(2p-6)(2p-7)} \leq 0.
	\end{split}
\end{equation}
The following example shows that inequality (\ref{Exampleeq}) does not hold for all $4 < p < 3 + \sqrt{7.5} \approx 5.73861$.
\begin{examplebold}\label{example1}
	Let $\alpha = 1$. Then condition (\ref{Condition}) does not hold when $4 < p \leq 5.1$ but it holds when $5.5 \leq p < 5.74$.
\end{examplebold}
	The argument is split into four cases. In the first three cases we find a lower bound for the left-hand side of inequality (\ref{Exampleeq}) and show that it is positive, proving that (\ref{Exampleeq}) cannot hold. For the last case we find an upper bound for the left-hand side of (\ref{Exampleeq}) and show that it is negative, proving that (\ref{Exampleeq}) holds.
	\begin{enumerate}
		\item[\textbf{(i)}] $\mathbf{4 < p \leq 4.5}$. By using Lemma \ref{11} we have
			\begin{equation*}
				B\left(\frac{3}{p}, 2p-8\right) \geq \frac{1}{2p-8}+\frac{p}{3}-1.
			\end{equation*}
			Utilizing this and denoting $J(p)=\frac{2(p-4)(p-2)}{B\left(\frac{3}{p},1-\frac{3}{p}\right)}$, we get a lower estimate for (\ref{Exampleeq}):
			\begin{equation*}
				\begin{split}
					J(p)F_{1,p}(0) &\geq \frac{1}{2}(-p^2+6p)\left(\frac{1}{2p-8}+\frac{p}{3}-1\right) -\frac{p-2}{p-4}\\
							&+\frac{(2p-5+3/p)(2p-6+3/p)(2p-7+3/p)(2p-8+3/p)}{2(p-2)(2p-5)(2p-6)(2p-7)}\\
							&=\frac{1}{2}(-p^2+6p)\frac{2p^2-14p+27}{6(p-4)}-\frac{p-2}{p-4}\\
							&+\frac{(2p-5+3/p)(2p-6+3/p)(2p-7+3/p)(2p-8+3/p)}{2(p-2)(2p-5)(2p-6)(2p-7)} :=S(p).
				\end{split}
			\end{equation*}
			Define $f(p) = \frac{1}{12}(-6-39p+18p^2-2p^3)$. Then $f(p) \geq f(4) = -\frac{1}{6}$. For the remaining term we obtain
			\begin{equation*}
				\begin{split}
					S(p) &=f(p) + \frac{(2p-8+3/p)}{2(p-2)}\left(1+\frac{3}{(2p-5)p}\right)\left(1+\frac{3}{(2p-6)p}\right)\left(1+\frac{3}{(2p-7)p}\right).
				\end{split}
			\end{equation*}
			Now we get a lower bound for $S(p)$, namely
			\begin{equation*}
				\begin{split}
					 	&f(4) + \frac{(2\cdot4-8+3/4)}{2(4.5-2)}\left(1+\frac{3}{(2\cdot4.5-5)\cdot4.5}\right)\\
						&\left(1+\frac{3}{(2\cdot4.5-6)\cdot4.5}\right)\left(1+\frac{3}{(2\cdot4.5-7)\cdot4.5}\right) \approx -\frac{1}{6} + 0.285185 > 0.
				\end{split}
			\end{equation*}
			So we have that $F_{1,p}(0) \geq \frac{1}{J(p)}S(p) > 0$ for all $4 < p \leq 4.5$.
		\item[\textbf{(ii)}] $\mathbf{4.5 < p \leq 4.9}$.
			Again by Lemma \ref{lower bound1} we get
			\begin{equation*}
				B\left(\frac{3}{p}, 2p-8\right) \geq \frac{p}{6(p-4)}\frac{2p^2-8p+3}{7p-24}.
			\end{equation*}
			We now get a lower estimate for $J(p)F_{1,p}(0)$:
			\begin{equation*}
				\begin{split}
						&\frac{144-78p-3p^2+12p^3-2p^4}{-288+84p}\\
						&+\frac{2p-8+3/p}{2(p-2)}\left(1+\frac{3}{(2p-5)p}\right)\left(1+\frac{3}{(2p-6)p}\right)\left(1+\frac{3}{(2p-7)p}\right)\\
						&= f(p) + \frac{(2p-8+3/p)}{2(p-2)}\left(1+\frac{3}{(2p-5)p}\right)\left(1+\frac{3}{(2p-6)p}\right)\left(1+\frac{3}{(2p-7)p}\right),
				\end{split}
			\end{equation*}
			where
			\begin{equation*}
				 f(p) = \frac{144-78p-3p^2+12p^3-2p^4}{-288+84p}.
			\end{equation*}
			The function $f(p)$ is non-increasing on the interval $4.5 < p \leq 4.9$, so it holds that $f(p) \geq f(4.9)\approx -0.415875$. For the second term we have
			\begin{equation*}
				\begin{split}
					&\frac{(2p-8+3/p)}{2(p-2)}\left(1+\frac{3}{(2p-5)p}\right)\left(1+\frac{3}{(2p-6)p}\right)\left(1+\frac{3}{(2p-7)p}\right) \geq\\
					&\frac{(2\cdot4.5-8+3/4.5)}{2(4.9-2)}\left(1+\frac{3}{(2\cdot4.9-5)\cdot4.9}\right)\left(1+\frac{3}{(2\cdot4.9-6)\cdot4.9}\right)\\
					&\left(1+\frac{3}{(2\cdot4.9-7)\cdot4.9}\right)\approx 0.458474.
				\end{split}
			\end{equation*}
			Combining these two we get $F_{1,p}(0) \geq \frac{0.042599}{J(p)} >0$ %=0.01054975 > 0$ 
			for $4.5 < p \leq 4.9$.
		\item[\textbf{(iii)}] $\mathbf{4.9 < p \leq 5.1}$. In this case we use the same estimate as in \textbf{(ii)}. Now we get
			\begin{equation*}
				\begin{split}
					J(p)F_{1,p}(0) &\geq (2p-8+3/p)\Bigg[\frac{(144-78p-3p^2+12p^3-2p^4)}{(-288+84p)(2p-8+3/p)}\\
							&+\frac{1}{2(p-2)}\left(1+\frac{3}{(2p-5)p}\right)\left(1+\frac{3}{(2p-6)p}\right)\left(1+\frac{3}{(2p-7)p}\right)\Bigg],
				\end{split}
			\end{equation*}
			where $(2p-8+3/p) > 0$ and 
			\begin{equation*}
				\begin{split}
					g(p) &:=\frac{1}{2(p-2)}\left(1+\frac{3}{(2p-5)p}\right)\left(1+\frac{3}{(2p-6)p}\right)\left(1+\frac{3}{(2p-7)p}\right)\\
						& \geq g(5.1)\approx 0.242306.
				\end{split}
			\end{equation*}
			So it remains to be shown that
			\begin{equation*}
				h(p):=\frac{(144-78p-3p^2+12p^3-2p^4)}{(-288+84p)(2p-8+3/p)}
			\end{equation*}
			is a non-increasing function on the interval $4.9 < p < 5.1$. By taking the derivative we have
			\begin{equation*}
				h'(p) = -\frac{5184-5616p+495p^2+4383p^3-4896p^4+2100p^5-396p^6+28p^7}{6(24-7p)^2(3-8p+2p^2)^2}.
			\end{equation*}
			Since the denominator is positive it remains to show that the nominator is positive on the aforementioned interval. Set
			\begin{equation*}
				g_1(p) = 5184-5616p+495p^2+4383p^3-4896p^4+2100p^5-396p^6+28p^7.
			\end{equation*}
			Then
			\begin{equation*}
				\begin{split}
					&g_1'(p) = -5616 + 990p+13149p^2-19584p^3+10500p^4-2376p^5+196p^6;\\
					&g_1''(p) = 990 + 26298p-58752p^2+42000p^3-11880p^4+1176p^5;\\
					&g_1'''(p) = 26298-117504p+126000p^2-47520p^3+5880p^4;\\
					&g_1^{(4)}(p) = -117504+252000p-142560p^2+23520p^3;\\
					&g_1^{(5)}(p) = 252000-284120p+70560p^2;\\
					&g_1^{(6)}(p) = -285120 + 141120p > g_1^{(6)}(4.9) > 0.
				\end{split}
			\end{equation*}
			Hence, $g_1^{(5)}(p)$ is non-decreasing and we have $g_1^{(5)}(p) \geq g_1^{(5)}(4.9) > 0$. Therefore, $g_1^{(4)}(p)$ is non-decreasing and so $g_1^{(4)}(p) \geq g_1^{(4)}(4.9) > 0$. Continuing like this we finally obtain that $g_1'(p)$ is non-decreasing and $g_1'(4.9) > 0$. Hence $g_1(p) \geq g_1(4.9) > 0$. Thus $h'(p) < 0$ and we get that 					$h(p) \geq h(5.1) \approx -0.237716$. Combining our results we now have
			\begin{equation*}
				F_{1,p}(0) \geq \frac{1}{J(p)}(h(5.1) + g(5.1)) \approx \frac{0.004590}{J(p)} > 0. %1.1475\cdot10^{-3} > 0.
			\end{equation*}
		\item[\textbf{(iv)}] $\mathbf{5.5 \leq p < 5.74}$. By plotting inequality (\ref{Exampleeq}) in Mathematica it is shown to hold for the values $5.25 \leq p < 5.74$. It is, however, difficult to prove for all these values and therefore we will prove it for the values $5.5 \leq p < 5.74$. We begin by using
			\begin{equation*}
				B\left(\frac{3}{p}, 2p-8\right) \leq \frac{1}{2p-8}+\frac{p}{3}-1.
			\end{equation*}
			In the same way as in \textbf{(i)} we now get the upper estimate:
			\begin{equation*}
				\begin{split}
					J(p)F_{1,p}(0)&\leq \frac{1}{2}(-p^2+6p)\left(\frac{1}{2p-8}+\frac{p}{3}-1\right) -\frac{p-2}{p-4}\\
							&+\frac{(2p-5+3/p)(2p-6+3/p)(2p-7+3/p)(2p-8+3/p)}{2(p-2)(2p-5)(2p-6)(2p-7)}\\
							&=\frac{1}{2}(-p^2+6p)\frac{2p^2-14p+27}{6(p-4)}-\frac{p-2}{p-4}\\
							&+\frac{(2p-5+3/p)(2p-6+3/p)(2p-7+3/p)(2p-8+3/p)}{2(p-2)(2p-5)(2p-6)(2p-7)}.
				\end{split}
			\end{equation*}
			Set 
			\begin{equation*}
				\begin{split}
					f(p) &= \frac{1}{2}(-p^2+6p)\left(\frac{1}{2p-8}+\frac{p}{3}-1\right) -\frac{p-2}{p-4}\\
						&= \frac{1}{12}(-6-39p+18p^2-2p^3).
				\end{split}
			\end{equation*}
			Now we write
			\begin{equation*}
				S(p) =f(p) + \frac{(2p-8+3/p)}{2(p-2)}\left(1+\frac{3}{(2p-5)p}\right)\left(1+\frac{3}{(2p-6)p}\right)\left(1+\frac{3}{(2p-7)p}\right).
			\end{equation*}
			By \textbf{(i)} we see that $f$ is non-increasing on the interval $5.5 < p < 5.74$. By dividing the interval $[5.5, 5.74]$ into two intervals $I_1 = [5.5, 5.62]$ and $I_2 = [5.62, 5.74]$, we show that $S(p) \leq 0$. First we consider $I_2$:
			\begin{equation*}
				\begin{split}
					S(p) &\leq f(5.62) + \frac{(2\cdot5.74-8+3/5.74)}{2(5.62-2)}\left(1+\frac{3}{(2\cdot5.62-5)\cdot5.62}\right)\\
						&\left(1+\frac{3}{(2\cdot5.62)-6)\cdot5.62}\right)\left(1+\frac{3}{(2\cdot5.62)-7)\cdot5.62}\right)\\
						&\approx -0.227916 < 0.
				\end{split}
			\end{equation*}
			Now for $I_1$ we further divide it into two intervals $I'_1 = [5.5,5.56]$ and $I''_1 = [5.56, 5.62]$. For $I''_1$ we have
			\begin{equation*}
				\begin{split}
					S(p) &\leq f(5.56) + \frac{(2\cdot5.62-8+3/5.62)}{2(5.56-2)}\left(1+\frac{3}{(2\cdot5.56-5)\cdot5.56}\right)\\
						&\left(1+\frac{3}{(2\cdot5.56)-6)\cdot5.56}\right)\left(1+\frac{3}{(2\cdot5.56)-7)\cdot5.56}\right)\\
						&\approx -0.125168 < 0.
				\end{split}
			\end{equation*}
			Finally, for the interval $I'_1$ we have
			\begin{equation*}
				\begin{split}
					S(p) &\leq f(5.5) + \frac{(2\cdot5.56-8+3/5.56)}{2(5.5-2)}\left(1+\frac{3}{(2\cdot5.5-5)\cdot5.5}\right)\\
						&\left(1+\frac{3}{(2\cdot5.5)-6)\cdot5.5}\right)\left(1+\frac{3}{(2\cdot5.5)-7)\cdot5.5}\right)\\
						&\approx -0.010372 < 0,
				\end{split}
			\end{equation*}
			and so $S(p) < 0$ on $5.5 < p < 5.74$. Since $F_{1,p}(0) \leq \frac{1}{J(p)}S(p)$, we see that (\ref{Exampleeq}) holds.
	\end{enumerate}

\noindent \textbf{Acknowledgements.} The first and the second authors were  partially supported by the Academy of Finland project 296718. The last author is grateful for the financial support from the Doctoral Network in Information Technologies and Mathematics at \AA bo Akademi University.


\begin{thebibliography}{99}

\bibitem{1} A. Aleman, Alexandru, A. Montes--Rodr\'{i}guez, and A.  Sarafoleanu, \emph{The eigenfunctions of the Hilbert matrix}, 
Constr. Approx. 36 (2012),  353–-374. 

\bibitem{2} M. Abramowitz, I. A. Stegun, \emph{Handbook of mathematical functions, with formulas, graphs, and mathematical tables}, National bureau of Standards Applied mathematics Series, vol. 55 (1965).

\bibitem{5} B. Bhayo, J. S\'andor, \emph{On the inequalities for beta function}, Notes on Number Theory and Discrete Mathematics vol. 21 (2015), 1-7.

\bibitem{4} V. Bo\v{z}in, B. Karapetrovi\'{c}, \emph{Norm of the Hilbert matrix on Bergman spaces}, J. Funct. Anal. 274 (2018), 525-543.

\bibitem{10} E. Diamantopoulos, \emph{Hilbert matrix on {B}ergman spaces}, Illinois J. Math. 48  (2004), 1067-1078.

\bibitem{1} E. Diamantopoulos, A. G. Siskakis, \emph{Composition operators and the Hilbert matrix}, Studia Mathematica 140  (2000), 191-198.

\bibitem{7} M. Dostani\'{c}, M. Jevti\'{c}, D. Vukoti\'{c}, \emph{Norm of the Hilbert matrix on Bergman and Hardy spaces and theorem of Nehari type}, J. Funct. Anal.  254  (2008), 2800-2815.

\bibitem{11} P. Iv\'{a}dy, \emph{On a beta function inequality}, J. Math. Inequal.  6 (2012), 333-341.

\bibitem{12}  M. Jevti\'c, D. Vukoti\'c, M. Arsenovi\'c, Taylor coefficients and coefficient multipliers of Hardy and Bergman-type spaces,  RSME Springer Series, 2. Springer, 2016.

\bibitem{6} B. Karapetrovi\'{c},  \emph{Norm of the Hilbert matrix operator on the weighted Bergman spaces}, Glasgow Mathematical Journal  60 (2018), 513-525.

\bibitem{8} M. Lindstr\"{o}m, S. Miihkinen, N. Wikman, \emph{Norm estimates of weighted composition operators pertaining to the Hilbert matrix}, To appear in Proc.~Amer.~Math.~Soc.

\bibitem{3} W. Magnus, \emph{On the spectrum of Hilbert's matrix}, Amer. J. Math. 72 (1950), 699-704.


\bibitem{9} V.V. Prasolov, \emph{Polynomials}, Algorithms Comput. Math., vol. 11, Springer, Berlin (2010).

\end{thebibliography}
\end{document}